\documentclass[10pt,oneside,reqno]{amsart}

\usepackage{amsmath,amsthm} 
\usepackage{enumerate}
\usepackage{xfrac}          
\numberwithin{equation}{section}
\theoremstyle{definition}
\newtheorem{Definition}{Definition}[section]
\newtheorem{Example}[Definition]{Example}
\newtheorem{Remark}[Definition]{Remark}

\theoremstyle{plain}
\newtheorem{Theorem}[Definition]{Theorem}

\newtheorem{Proposition}[Definition]{Proposition}
\newtheorem{Corollary}[Definition]{Corollary}
\newtheorem{Lemma}[Definition]{Lemma}

\usepackage{amssymb}        
\usepackage{mathrsfs}       
\usepackage{eucal}          
\usepackage{stmaryrd}       

\usepackage{geometry}
\usepackage{lscape}

\usepackage{hyperref}

\usepackage[usenames,dvipsnames]{xcolor}
\usepackage{tikz}
\usetikzlibrary{arrows, matrix}
\usepackage{tikz-cd}
\usepackage{subfig}         
\usepackage{arydshln}       
\input xy
\xyoption{all}

\newcommand{\al}{\alpha}

\newcommand{\Ga}{\Gamma}

\newcommand{\la}{\lambda}

\newcommand{\si}{\sigma}

\newcommand{\om}{\omega}

\newcommand{\N}{\mathbb{N}}
\newcommand{\Z}{\mathbb{Z}}

\newcommand{\C}{\mathbb{C}}
\newcommand{\K}{\Bbbk}

\newcommand{\Fgl}{\mathfrak{gl}}
\newcommand{\Fsl}{\mathfrak{sl}}
\newcommand{\Fsp}{\mathfrak{sp}}

\newcommand{\Fg}{\mathfrak{g}}

\newcommand{\Fm}{\mathfrak{m}}
\newcommand{\mf}{\mathfrak{m}}

\newcommand{\scA}{\mathscr{A}}

\newcommand{\op}{\operatorname}

\DeclareMathOperator{\Aut}{Aut}

\DeclareMathOperator{\Ext}{Ext}

\DeclareMathOperator{\Id}{Id}
\renewcommand{\Im}{\operatorname{Im}}

\DeclareMathOperator{\MaxSpec}{MaxSpec}

\DeclareMathOperator{\Supp}{Supp}

\renewcommand{\hat}{\widehat}

\renewcommand{\mod}[1]{\;\text{(mod $#1$)}}

\newcommand{\ronto}{\twoheadrightarrow}

\renewcommand{\bar}{\overline}

\title[Grothendieck rings of towers of GWAs for finite orbits]{Grothendieck rings of towers of generalized Weyl algebras in the finite orbit case}
\author{Jonas T. Hartwig \and Daniele Rosso}
\date{\today}

\address{Department of Mathematics, Iowa State University, Ames, IA-50011, USA}
\email{jth@iastate.edu}
\urladdr{http://jthartwig.net}

\address{Department of Mathematics and Actuarial Science, Indiana University Northwest, Gary, IN-46408, USA}
\email{drosso@iu.edu}
\urladdr{https://drosso.pages.iu.edu/}

\begin{document}
\maketitle
\begin{abstract}
Previously we showed that the tensor product of a weight module over a generalized Weyl algebra (GWA) with a weight module over another GWA is a weight module over a third GWA. In this paper we compute tensor products of simple and indecomposable weight modules over generalized Weyl algebras supported on a finite orbit.
This allows us to give a complete presentation by generators and relations of the Grothendieck ring of the categories of weight modules over a tower of generalized Weyl algebras in this setting. We also obtain partial results about the split Grothendieck ring.
We described the case of infinite orbits in previous work.
\end{abstract}

\section{Introduction}

The class of noncommutative rings known as \emph{generalized Weyl algebras (GWAs)} was defined by Bavula \cite{Bavula1991}. 
They are denoted $R(\si,t)$, where $R$ is a base ring, $\si$ is an automorphism of $R$, and $t$ is an element in the center of $R$. On the one hand, this class contains many important algebras from representation theory such as $U(\Fg)$ and $U_q(\Fg)$ for $\Fg=\Fgl_2,\Fsl_2,\Fsl_3^+$, as well as the first Weyl algebra, down-up algebras \cite{BenkartRoby1998}, and many others. On the other hand, questions may be answered quantitatively for these algebras, from ring theoretical properties (global dimension, primitive ideals, automorphism groups, etc) to representation theory problems (classifying simple modules \cite{Bavula1992,BavulaVanOystaeyen2004}, classifying and computing explicit structure constants for all indecomposable weight modules \cite{DGO}).
They also appeared from a different perspective in \cite{Rosenberg1995}.
There is also a rank $n$ generalization of GWAs \cite{Bavula1992}, and a further ``more noncommutative'' generalization of the higher rank GWAs called \emph{twisted generalized Weyl algebras} (TGWAs) \cite{MazorchukTurowska1999}.

In another direction, GWAs can also be considered as special cases of Bell-Rogalski algebras, as explained in \cite{GadRosWon2022}. This family of rings was introduced in \cite{BellRog2016} with an aim to classify $\mathbb{Z}$-graded simple rings that are birationally commutative, and the categories of their graded modules should have interesting connections to noncommutative geometry, which gives further motivation for studying these algebras.
In noncommutative geometry, noncommutative rings stand in for (commutative) algebras of functions on varieties: for example noncommutative Kleinian singularities of type $A$ (\cite{Hodges1993}) are examples of GWAs.

In \cite{HR} we found a new interesting structure for TGWAs: the tensor product over the base ring $R$ of two weight modules for two (potentially different) TGWAs is naturally a weight module for another TGWA. In the simplest non-trivial case, the tensor product of two modules over the first Weyl algebra becomes a module over $U(\Fsl_2)$, and by varying the presentation of the Weyl algebra as a GWA, any Verma module for $\Fsl_2$ can be obtained this way (see \cite[Example 3.10]{HR}). We then have a tensor structure on the direct sum of categories of weight modules, where the sum ranges over a family (or ``tower") of TGWAs. A natural way to try to understand this new tensor category is to describe its Grothendieck ring.

The category of weight modules over a (T)GWA breaks into a direct sum over orbits in the maximal spectrum of $R$ with respect to the action of a free abelian group. (Usually $R$ is assumed commutative when studying weight modules.)
By the orbit-stabilizer theorem, the orbit is isomorphic, as a $\Z^n$-set, to $\Z^n/H$ where $H$ is the stabilizer of a point. 
Thus, in rank $n=1$, there are two cases: $H=\{0\}$ which leads to an infinite orbit, and $H\neq \{0\}$, which corresponds to a finite orbit. As an example, for weight modules over $U(\Fsl_2)$ realized as a GWA, either all orbits are finite (in positive characteristic) or all orbits are infinite (in characteristic zero). Similarly, for $U_q(\Fsl_2)$ the size of the orbit depends on whether $q$ is a root of unity or not.
Thus, the infinite orbit case is easier to study (for example, there are no non-trivial self-extensions between simple weight modules in this case).

In \cite{HR} we gave a complete presentation of the Grothendieck ring and split Grothendieck ring of the categories of weight modules for a tower of TGWAs of rank $1$ with support in a fixed infinite orbit. In this paper we address the finite orbit case.

Our main result is Theorem \ref{thm:groth-ring}, giving a presentation by generators and relations of $\mathscr{A}(\omega,\mathsf{M})$, which is the Grothendieck ring of the direct sum of categories of weight modules over a tower of GWAs, with support in a finite orbit of order $p\in\Z_{>0}$. Notice that one important difference from the infinite orbit case is that here in the finite orbit case we need to include some generators in degree $2$, while in \cite[Theorem 4.6]{HR} the Grothendieck ring was generated in degree $1$.

The other results are about the split Grothendieck ring. Since that is much more complicated we are only able to obtain partial results, by considering smaller towers of GWAs. In Theorem \ref{thm:nobr-split-ring} we give a presentation of the split Grothendieck ring when we restrict to a single parameter $t=1$, which shows a connection between the representation theories of $\Fsl_2$ and of the algebra of skew Laurent polynomials. In Theorem \ref{thm:split-quotient} we describe the quotient $\mathscr{A}^{\mathrm{split}}(\omega,\mathsf{M}_0)/\mathscr{I}(\om,\mathsf{M}_0)$, where the parameters $t$ are taken to be in the monoid $\mathsf{M}_0$ consisting of powers of a single linear polynomial, and the ideal $\mathscr{I}(\om,\mathsf{M}_0)$ is spanned in a natural way by one of the two families of indecomposable weight modules. We also show in Theorem \ref{thm:section} that, in the case of the monoid $\mathsf{M}_0$, the natural surjective map $\mathscr{A}^{\mathrm{split}}(\omega,\mathsf{M}_0)\ronto\mathscr{A}(\omega,\mathsf{M}_0)$ has a section which is an algebra map (equivalently, tensor products of semisimple weight modules over GWAs with $t\in\mathsf{M}_0$ are semisimple).

The structure of this paper is as follows. 

\begin{itemize}
    \item In Section \ref{sec:notation} we recall some known results that we will need: we first review the definitions of the tensor product of weight modules and of the (split) Grothendieck rings, from \cite{HR}, then we describe the simple and indecomposable weight modules over GWAs by adapting the classification from \cite{DGO}. We also discuss a connection to quiver algebras.
    \item In Section \ref{sec:definitions}  we compute tensor products over the base ring $R$ of the simple and indecomposable weight modules for a GWA, supported on a finite orbit.
    \item In Section \ref{sec:Groth-ring} we use the results from the previous section to prove the main result, Theorem \ref{thm:groth-ring}.
    \item In Section \ref{sec:split-groth-ring} we obtain some partial results about split Grothendieck rings, in particular the main result of this section is Theorem \ref{thm:split-quotient}.
    \item In Section \ref{sec:graph-mod} we give an alternative way to define some of the weight modules and compute their tensor products, using certain directed graphs.
\end{itemize}

\subsection{Future Directions}
The tensor product operation defined in \cite{HR} is very general, so there are many open questions about what these Grothendieck rings look like in several specific cases that have not been studied yet. For example, in \cite{HR} we computed the Grothendieck group for a tower or rank $2$ algebras with a relation between the automorphisms which makes for a cylindrical orbit, however we were not able to say anything about the split Grothendieck ring because we would first need to classify the indecomposable weight modules in that setting. Similarly, one could ask about the case of a tower of rank $2$ TGWAs with two finite order automorphisms (torus orbit). In that case not much is known about the weight modules, not even the simple ones, so significant work is needed before describing those Grothendieck rings.

In a different direction, one may ask if all simple completely pointed weight modules over $U(\Fgl_n)$ and $U(\Fsp_{2n})$
(studied from the perspective of TGWAs in \cite{HarSer})
can be realized as tensor products of weight modules over Weyl algebras and Laurent polynomials rings.

\subsection*{Acknowledgments}

D. R. was supported in this work by a Summer Faculty Fellowship for research of Indiana University Northwest.
J. H. was supported by Simons Foundation Collaboration Grant \#637600.

\section{Background and Notation}\label{sec:notation}

\subsection{Generalized Weyl algebras and Grothendieck rings}Here we recall some necessary definitions and results from \cite{HR}, although we simplify them for our current setting.

\begin{Definition}[\cite{Bavula1992}]\label{def:tgwa} 
Let $R$ be a commutative ring, $\si\in \Aut(R)$ an automorphism of $R$, $t\in R$ a regular element. The \emph{generalized Weyl algebra (GWA)} of rank (or degree) one, denoted $A(R,\si,t)$, is the associative algebra obtained from $R$ by adjoining $2$ new generators
$X, Y$ that are subject to the following relations, for all $r\in R$:
\begin{equation}\label{eq:TGWA-Rels}
X  r = \si(r) X,\qquad
Y  r = \si^{-1}(r) Y,\qquad 
Y X =t, \qquad 
X Y =\si(t).
\end{equation}  
\end{Definition}

We now fix $R$ and $\si$, and we allow $t\in R_{\operatorname{reg}}$ to vary, so we denote
$A(t):=A(R,\si,t)$. Let $A(t)\otimes_R A(t')$ denote the tensor product of $A(t)$ and $A(t')$ viewed as left modules over the commutative ring $R$. Explicitly:
\begin{equation}\label{eq:Rtensor}
A(t)\otimes_R A(t') := \big(A(t)\otimes_\Z A(t')\big)/\big((ra)\otimes b - a\otimes (rb)\mid r\in R,~ a\in A(t),~ b\in A(t')\big).
\end{equation}

\begin{Theorem}[See {\cite[Thm 3.3]{HR}} for the more general statement]\label{thm:coprod}
\begin{enumerate}[{\rm (a)}]
\item For any two regular elements $t, t'\in R$, there is a homomorphism of $R$-rings
\begin{equation}\label{eq:tgwa-coprod}
\Delta_{t,t'}:A(tt')\to A(t)\otimes_R A(t')
\end{equation}
which is uniquely determined by
\begin{align*} r\mapsto & r\otimes 1=1\otimes r\qquad \forall r\in R, \\
X(tt')\mapsto & X(t)\otimes X(t'), \\ 
Y(tt')\mapsto & Y(t)\otimes Y(t'),
\end{align*} 
where $X(t), Y(t)$ denote the generators in $A(t)$.
\item For any three regular elements $t, t', t''\in R$, the following coassociative law holds:
\begin{equation}
\big(\Delta_{t,t'}\otimes \Id_{A(t'')}\big)\circ\Delta_{tt',t''} = 
\big(\Id_{A(t)}\otimes \Delta_{t,t't''}\big)\circ\Delta_{t,t't''}
\end{equation}
\item For any two regular elements $t, t'\in R$, the following cocommutative law holds:
\begin{equation}
P\circ \Delta_{t,t'} = \Delta_{t',t},
\end{equation}
where $P(x\otimes y)=y\otimes x$.
\end{enumerate}
\end{Theorem}
\begin{Definition}
If $A$ is an $R$-ring, an $A$-module $M$ is called a  \emph{weight module} if
\[
M=\bigoplus_{\Fm\in\MaxSpec(R)} M_\Fm,\qquad
M_\Fm=\{v\in M\mid \Fm v=0\}.
\]
For a weight module $M$, we define the \emph{support} of $M$ to be
$$ \Supp(M):=\{\Fm\in\MaxSpec(R)~|~M_\Fm\neq 0\}.$$
For $t\in R$ regular, we let $A(t)\text{-wmod}$ be the category of weight modules $M$ for $A(t)$ such that $M_\Fm$ is a finite dimensional $R/\Fm$ vector space for all $\Fm\in\MaxSpec(R)$.
\end{Definition}

We let $M\otimes_R M'$ denote the tensor product of left $R$-modules over the commutative ring $R$.

\begin{Lemma}[See {\cite[Lemma 3.8]{HR}}]\label{lemma:weightmod}
If $M$ is a weight module for $A(t)$ and $M'$ is a weight module for $A(t')$, then $M\otimes_R M'$ is a weight module for $A(t)\otimes_R A(t')$ via the action $(a\otimes a')\cdot (v\otimes v')=av\otimes a'v'$. In particular we have $(M\otimes_R M')_\Fm=M_\Fm\otimes_R M'_\Fm$ for all $\Fm\in\MaxSpec(R)$ and $\Supp(M\otimes_R M')=\Supp(M)\cap \Supp(M')$.
\end{Lemma}

\begin{Definition}
Let $\mathsf{M}\subset R_{\operatorname{reg}}$ be a multiplicatively closed subset of the regular elements of $R$, then we can define two $\mathsf{M}$-graded $\C$-algebras
$$\mathscr{A}(\mathsf{M})=\bigoplus_{t\in\mathsf{M}} \C\otimes_{\Z}K_0(A(t)\text{-wmod})\qquad\qquad \mathscr{A}^{\text{split}}(\mathsf{M})=\bigoplus_{t\in\mathsf{M}} \C\otimes_{\Z}K_0^{\text{split}}(A(t)\text{-wmod})$$
with multiplication given by the map $$[M]\otimes [M']\mapsto [M\otimes_{R} M']$$
where $M\otimes_{R} M'$ is seen as a module for $A(tt')$ by restriction under the map \eqref{eq:tgwa-coprod}.
\end{Definition}
\begin{Remark}
By Theorem \ref{thm:coprod}(b)(c), multiplication in $\mathscr{A}(\mathsf{M})$ and $\mathscr{A}^{\text{split}}(\mathsf{M})$ is associative and commutative. 
\end{Remark}

Let $\omega\in\MaxSpec(R)/\Z$ be an orbit in the maximal spectrum of $R$ with respect to the action of $\sigma$.
Let $A(t)\mathrm{-wmod}_{\omega}$ be the full subcategory of $A(t)\mathrm{-wmod}$ consisting of all weight modules with support in $\omega$. It is well-known (and easy to see from $XM_\Fm\subseteq M_{\si(\Fm)}, YM_\Fm\subseteq M_{\si^{-1}(\Fm)}$) that $A(t)\mathrm{-wmod}$ splits as follows:
\[
A(t)\mathrm{-wmod} = \prod_{\omega\in\MaxSpec(R)/\Z} A(t)\mathrm{-wmod}_\omega.
\]
As a consequence we have the following result.
\begin{Proposition}[{\cite[Prop. 3.15]{HR}}]\label{prop:direct-sum}We have the following direct products of $\C$-algebras:
\begin{equation}\label{eq:directsum} \mathscr{A}(\mathsf{M})=\prod_{\omega\in\MaxSpec(R)/\Z}\scA(\mathsf{M},\omega),\qquad \mathscr{A}^{\mathrm{split}}(\mathsf{M})=\prod_{\omega\in\MaxSpec(R)/\Z}\scA^{\mathrm{split}}(\mathsf{M},\omega) \end{equation}
where $\scA(\mathsf{M},\omega)=\bigoplus_{t\in\mathsf{M}} \C\otimes_{\Z}K_0(A(t)\text{-wmod}_\omega)$ and analogously for $\scA^{\mathrm{split}}(\mathsf{M},\omega)$.
\end{Proposition}

\subsection{Connection to quiver algebras}

Although not needed in the rest of the paper, we provide the following perspective relating weight modules over GWAs to quiver algebra representations.

Let $\omega\in\MaxSpec(R)/\Z$. If $|\omega|=p<\infty$ we assume that $\si^p$ is the identity on $R/\Fm$ for $\Fm\in\omega$. Fix $\Fm_0\in\omega$ and let $\Bbbk$ be the field $R/\Fm_0$. Let $\mathcal{Q}=(\mathcal{Q}_0,\mathcal{Q}_1,s,t)$ be the quiver with vertex set $\mathcal{Q}_0=\omega$ and arrow set $\mathcal{Q}_1=\{\al_\Fm^\pm\}_{\Fm\in\omega}$ with source $s(\al_\Fm^\pm)=\Fm$ and target $t(\al_\Fm^\pm)=\si^{\pm 1}(\Fm)$. 

\begin{Proposition} 
There is a tensor functor from the category $\bigoplus_{t\in \mathsf{M}} A(t)\mathrm{-wmod}_{\omega}$ to the category of $\Bbbk$-linear representations of $\mathcal{Q}$.
This functor induces a ring homomorphism from the
(split) Grothendieck ring $\mathscr{A}(\mathsf{M},\omega)$ to the (split) Grothendieck ring of $\Bbbk$-linear quiver representations of $\mathcal{Q}$.
\end{Proposition}

\begin{proof}
Let $V$ be an object of $A(t)\text{-wmod}_\omega$. We define a corresponding module $F(V)$ for the quiver algebra as follows. For $\Fm\in\omega$, let $F(V)_\Fm$ be the weight space $V_\Fm$.
All the orbit elements $\Fm\in\omega$ are of the form $\Fm=\si^i(\Fm_0)$ for some $i\in\Z$. We identify $\Bbbk$ with $R/\si^i(\Fm_0)$ via the map induced by $\si^i$.
In this way we can regard $V_\Fm$ as a vector space over $\Bbbk$.
Then, for an arrow $\al=\al_\Fm^+:\Fm\to \si(\Fm)$ in the quiver, we let the corresponding $\Bbbk$-linear transformation $x_\al$ be given by the action of $X$. Similarly for an arrow $\beta=\al_\Fm^-:\Fm\to\si^{-1}(\Fm)$ we define $x_\beta$ to be given by the action of $Y$. 
If $\varphi:V\to W$ is a morphism then $F(\varphi)_\Fm:V_\Fm \to W_\Fm$ is the restriction of $\varphi$ to the weight space $V_\Fm$ for each $\Fm$. Since all the weight space functors $V\mapsto V_{\Fm}$ are exact, $F$ is exact.

We show that the functor $F$ constructed is a tensor functor. By definition, $F(V\otimes_R V')_\Fm = (V\otimes_R V')_\Fm$. By Lemma \ref{lemma:weightmod}, $(V\otimes_R V')_\Fm=V_\Fm\otimes_\Bbbk V'_\Fm=F(V)_\Fm\otimes_\Bbbk F(V')_\Fm$. This shows that $F(V\otimes_R V')=F(V)\otimes_\K F(V')$ as vector spaces. Furthermore, by Theorem \ref{thm:coprod}(a), $X$ and $Y$ are acting diagonally on tensor products of weight spaces, which corresponds to the diagonal action of the maps $x_\al$ and $x_\beta$ from the quiver. By the definition of tensor products of quiver representations, this shows that $F(V\otimes V')=F(V)\otimes F(V')$.

Since the functor $F$ is an exact tensor functor, it is immediate that we obtain a ring homomorphism of (split) Grothendieck rings.
\end{proof}

\subsection{Weight modules over GWAs supported on finite orbits}

Let $t\in R$ be a regular element, $A(t)=A(R,\si,t)$ the corresponding GWA,  $\om\in\MaxSpec(R)/\Z$ be a fixed finite orbit, $|\om|=p$. In this section we recall the description of simple and indecomposable objects in the category $A(t)\mathrm{-wmod}_\omega$ 
. For simplicity, from now on we assume that $\si^p=\Id_R$. Simple and indecomposable weight modules for GWAs, both supported on infinite orbits and on finite orbits, have been classified, see Bavula \cite{Bav96} and Drozd, Guzner and Ovsienko \cite{DGO}. We will rely heavily on the classification of \cite{DGO}, however we will use different notation and conventions that are better adapted to our goals of describing the tensor products of weight modules in the next section.

\begin{Definition}
A maximal ideal $\mf\in\MaxSpec(R)$ is called a \emph{break} for $A(t)$ if $t\in\mf$.
Let $B^t_\om$ be the set of all breaks for $A(t)$ in $\om$: $B^t_\om=\{\mf\in\om~|~ t\in\mf\}$.
\end{Definition}

\begin{Definition}
We fix once and for all $\mf_0\in\om$, so that $\om=\{\mf_0,\si(\mf_0),\ldots,\si^{p-1}(\mf_0)\}$. We define $\K:=R/\mf_0$. Given $t\in R_{\operatorname{reg}}$, we define the set of indexed words $(w,j)$ $$\mathbf{D}^t_\om=\{(w,j)~|~ w=w_{j}w_{j+1}\cdots w_{j+\ell},~ j\in\Z,~\ell\in \N, ~ w_i\in\{1,x,y,0\}, ~w_i=1 \text{ iff }\si^i(\mf_0)\not\in B^t_\om\}.$$
\end{Definition}

By abuse of notation, whenever we write $w_jw_{j+1}\cdots w_{j+\ell}\in \mathbf{D}^t_\om$ we really mean $(w_jw_{j+1}\cdots w_{j+\ell},\,j)\in \mathbf{D}^t_\om$.

\begin{Remark}
Comparing with the notation in \cite{DGO}, we have $\mf_0=\mf(\om)$ (which is assumed to be a break there, we do not make that assumption); $\K=K_\om$, $\K[x,x^{-1}]=K_\om[x,x^{-1},\tau_\om]$, $\K[x]=K_\om[x,\tau_\om]$ (since $\tau_\om=\si^p=\Id_R$ by our simplifying assumption).  Also, the set of words $\mathbf{D}^t_\om$ differs from the set $\mathbf{D}$ (the free monoid in $x$ and $y$) which was used in \cite[\S 4.8]{DGO}. Any word $w\in \mathbf{D}$ starts at index $1$ and its letters only label the breaks in the orbit. In our case we are labeling all the elements of the orbit, but non-breaks (and only those) are labeled by the letter $1$.
\end{Remark}

\begin{Definition}\label{def:finorb-circle}
Let $w=w_1w_2\ldots w_{r p}$ be a word such that $(w,1)\in \mathbf{D}^t_\om$, and let $F\in\Aut_{\K}(\K^d)$ which we can think of as an invertible $d\times d$ matrix with entries in $\K$.
We define an $R$-module $$V^t(\om,w,F)=\bigoplus_{\substack{1\leq k\leq r p\\ 1\leq s\leq d}} \left(R/\si^k(\mf_0)\right) e^t_{k,s}$$ with $A(t)$ action given by
\begin{align}
\label{eq:newVomwFX}
Xe^t_{k,s}&=
\begin{cases}
\si(t) e^t_{k+1,s},&\text{if $k\neq r p$, $w_k=1$},\\
e^t_{k+1,s},&\text{if $k\neq r p$, $w_k=x$},\\
\si(t)\sigma(F)e^{t}_{1,s},&\text{if $k= r p$, $w_{r p}=1$},\\
\sigma(F)e^{t}_{1,s},&\text{if $k= r p$, $w_{r p}=x$},\\
0,&\text{ if $w_k\in\{y,0\}$};
\end{cases}\\
\label{eq:newVomwFY}
Ye^t_{k,s}&=
\begin{cases}
e^t_{k-1,s},&\text{if $k\neq 1$, $w_{k-1}\in\{1,y\}$},\\
F^{-1}e^t_{\ell p,s},&\text{if $k=1$, $w_{r p}\in\{1,y\}$},\\
0,&\text{ if $w_{k-1}\in\{x,0\}$},
\end{cases}
\end{align}
where $\si(F)$ is the automorphism of $(R/\si(\Fm_0))^d$ induced by $F$ and we identified $\K^d=\oplus_{s=1}^d \K e^t_{r p,s}$. Furthermore, for $r\in R$ the expression $r e^t_{k,s}$ indicates the left $R$-module action of $r$ on the vector $e^t_{k,s}$.

If $f\in\K[x,x^{-1}]$ or $f\in\K[x]$, we define $V^t(\om,w,f):=V^t(\om,w,F_f)$ where $F_f$ is the \emph{companion matrix} of $f$, which is defined as follows, if $f=\al_0+\al_1x+\cdots+\al_dx^d$ with $\al_d\neq 0$, 
\[F_f=\begin{bmatrix}
0&0&0&\cdots&0&-\al_0/\al_d\\
1&0&0&\cdots&0&-\al_1/\al_d\\
0&1&0&\cdots&0&-\al_2/\al_d\\
\vdots&\vdots&\vdots&\ddots&\vdots&\vdots\\
0&0&0&\cdots&1&-\al_{d-1}/\al_d
\end{bmatrix}.
\] 
\end{Definition}

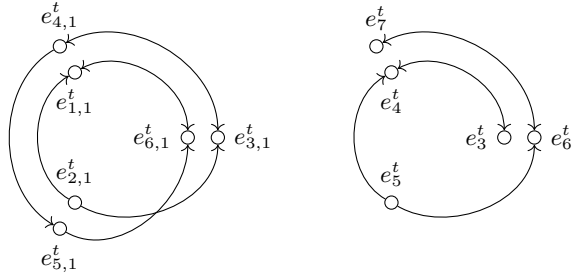
\begin{figure}
\centering
\begin{tikzpicture}[baseline=-.5ex, vertex/.style={draw,circle,fill=white,minimum size=5pt, inner sep=0pt}]
\node[vertex] (A) at (1,0) [label=left:\small $e^t_{6,1}$] {};
\node[vertex] (B) at ({cos(120)},{sin(120)}) [label=below:\small $e^t_{1,1}$] {};
\node[vertex] (C) at ({cos(240)},{sin(240)}) [label=above:\small $e^t_{2,1}$] {};
\node[vertex] (D) at (1.4,0) [label=right:\small $e^t_{3,1}$] {};
\node[vertex] (E) at ({1.4*cos(120)},{1.4*sin(120)}) [label=above:\small $e^t_{4,1}$] {};
\node[vertex] (F) at ({1.4*cos(240)},{1.4*sin(240)}) [label=below:\small $e^t_{5,1}$] {};
\draw[<->] (A) edge[out=90, in=30] (B);
\draw[<-] (B) edge[out=210, in=150] (C);
\draw[->] (C) edge[out=330, in=270] (D);
\draw[<->] (D) edge[out=90, in=30] (E);
\draw[->] (E) edge[out=210, in=150] (F);
\draw[->] (F) edge[out=330, in=270] (A);
\end{tikzpicture}
\qquad
\begin{tikzpicture}[baseline=-.5ex, vertex/.style={draw,circle,fill=white,minimum size=5pt, inner sep=0pt}]
\node[vertex] (A) at (1,0) [label=left:\small $e^t_{3}$] {};
\node[vertex] (B) at ({cos(120)},{sin(120)}) [label=below:\small $e^t_{4}$] {};
\node[vertex] (C) at ({cos(240)},{sin(240)}) [label=above:\small $e^t_{5}$] {};
\node[vertex] (D) at (1.4,0) [label=right:\small $e^t_{6}$] {};
\node[vertex] (E) at ({1.4*cos(120)},{1.4*sin(120)}) [label=above:\small $e^t_{7}$] {};
\draw[<->] (A) edge[out=90, in=30] (B);
\draw[<-] (B) edge[out=210, in=150] (C);
\draw[->] (C) edge[out=330, in=270] (D);
\draw[<->] (D) edge[out=90, in=30] (E);
\end{tikzpicture}
\caption{Illustration of the two kinds of modules. We take $p=3$, so $\om=\{\mf_0,\mf_1,\mf_2\}$, and $t$ such that $B^t_\om=\{\mf_1,\mf_2\}$. The vertices represent weight vectors in the basis, the weight space $V_{\mf_i}$, for $i=0,1,2$, consists of the vectors on the ray at an angle of $\frac{2\pi i}{3}$. The action of $X$ (resp. $Y$) moves counterclockwise (resp. clockwise) along the edges. The arrows on the edges indicate which action gives a nonzero result when acting on the given basis vectors, which correspond to the letters $1,x,y$ in the words $w$ (a letter $0$ would get no edge). The picture on the left corresponds to a module $V^t(\om,yx1xx1,F)$ from Definition \ref{def:finorb-circle} (with $d=1$ and an unspecified $F\in \Aut_\Bbbk(\Bbbk e^t_{6,1})$ ). The picture on the right corresponds to a module $V^t(\om,2,1yx1)$ from Definition \ref{def:type1mod}.}
\label{fig:goes-around-circle-twice}
\end{figure}

\begin{Remark}
If $0\neq W\subsetneq \K^d$ is an $F$-invariant subspace, then we get a nontrivial submodule $0\neq V^t(\om,w,F|_W)\subsetneq V^t(\om,w,F)$. If $\K^d=W_1\oplus W_2$ splits as a direct sum of $F$-invariant subspaces, we get submodule summands
$ V^t(\om,w,F)=V^t(\om,w,F|_{W_1})\oplus V^t(\om,w,F|_{W_2}).$
\end{Remark}

\begin{Remark}\label{rem:indec-nobr}
If $B^t_\om=\emptyset$, we necessarily have $w=1^{rp}$ and $V^t(\om,w,F)$ becomes a direct sum of indecomposable modules of the form $V(\om,f)$ as defined in \cite[§5.3]{DGO}. For any indecomposable polynomial $f\in\K[x,x^{-1}]$, we have an isomorphism $V(\om,f)\simeq V^t(\om,1^p,f)$. In particular, $V^t(\om,1^p,f)$ is an indecomposable (resp. simple) weight module if and only if $f$ is indecomposable (resp. irreducible) and all such modules are of this form.
\end{Remark}

\begin{Remark}\label{rem:hat-w}
If $B^t_\om\neq\emptyset$, a module $V(\om,w,f)$, as defined in \cite[§5.6]{DGO}, is isomorphic to a module $V^t(\om,\hat{w},f)$ of Definition \ref{def:finorb-circle} where $\hat{w}$ is obtained from $w$ by inserting the appropriate number of $1$'s in the word which correspond to the points in the orbit that are not breaks. In particular, $V^t(\om,w,f)$ is an indecomposable weight module if and only if the word $w$ is non-periodic (it is not the power of a word of length a multiple of $p$), it does not contain any letter $0$, and the polynomial $f$ is indecomposable. Similarly, $V^t(\om,w,f)$ is simple if and only if $f$ is irreducible and either $w=\hat{x^m}$ or $w=\hat{y^m}$ (where $m=|B^t_\om|$).
\end{Remark}

\begin{Definition}\label{def:type1mod}
Suppose $B_\omega^t\neq\emptyset$ and let $0\leq i\leq p-1$ be an index such that $\si^i(\mf_0)\in B^t_\om$, and $w=w_{i+1}w_{i+2}\cdots w_{i+\ell}$ such that $(w,i+1)\in\mathbf{D}^t_\om$, with $\si^{i+\ell+1}(\mf_0)\in B^t_\om$. We define an $R$-module by setting 
$$V^t(\om,i,w)=\bigoplus_{i+1\leq k\leq i+\ell+1} \left(R/\si^k(\mf_0)\right)e^t_k$$
with $A(t)$-action given by
\begin{align}
\label{eq:newVomiwX}
Xe^t_{k}&=
\begin{cases}
\si(t) e^t_{k+1},&\text{if $k\neq i+\ell+1$, $w_k=1$},\\
e^t_{k+1},&\text{if $k\neq i+\ell+1$, $w_k=x$},\\
0,&\text{if $k= i+\ell+1$ or if $w_k\in\{y,0\}$};
\end{cases}\\
\label{eq:newVomiwY}
Ye^t_{k}&=
\begin{cases}
e^t_{k-1},&\text{if $k\neq i+1$, $w_{k-1}\in\{1,y\}$},\\
0,&\text{if $k= i+1$ or if $w_{k-1}\in\{x,0\}$}.
\end{cases}
\end{align}
Notice that we allow $\ell=0$, that is, that $w=\emptyset$ is the empty word, if $\si^{i+1}(\mf_0)\in B^t_\om$. If that is the case, $V^t(\om,i,\emptyset)$ is one dimensional, and both $X$ and $Y$ act as zero.
\end{Definition}

\begin{Remark}
The modules $V(\om,j,w)$ defined in \cite[§5.5]{DGO} are isomorphic to $V^t(\om,i,\hat{w})$, where $\hat{w}$ is obtained from $w$ by inserting the appropriate number of $1$'s in the word which correspond to the points in the orbit that are not breaks. The corresponding index $i$ would then be such that $\si^i(\mf_0)$ is the $j$-th break in the orbit using the appropriate circular order. In particular, $V^t(\om,i,w)$ is an indecomposable weight module if and only if $w$ does not contain the letter zero, and $V^t(\om,i,w)$ is simple if and only if $w=1^s$ for some $s\geq 0$.
\end{Remark}

\begin{Remark}\label{rem:decomp-type1}
It is clear from the definition that if there exists $q$ such that $w_q=0$, then
$$ V^t(\om,i,w) \simeq V^t(\om,i,w')\oplus V^t(\om,\bar{q},w'')$$
where $w'=w_{i+1}w_{i+2}\cdots w_{q-1}$, $w''=w_{q+1}\cdots w_{i+\ell}$ and $0\leq \bar{q}\leq p-1$ is such that $\bar{q}\equiv q \pmod p$.
\end{Remark}

\begin{Lemma}\label{lemma:fakecirc}
If $F\in\Aut_{\K}(\K^d)$, $w=w_1\cdots w_{rp}$, and there exists $q$ such that $w_q=0$, then
\[V^t(\om,w,F)\simeq V^t(\om,\bar{q},w')^{\oplus d}\]
where $w'=w_{q+1}\cdots w_{rp} w_1 w_2\cdots w_{q-1}$ and $0\leq \bar{q}\leq p-1$ is such that $\bar{q}\equiv q \pmod p$.
\end{Lemma}

\begin{proof}
Given the basis $\{e^t_{k,s}\}$ of $V^t(\om,w,F)$ as in Definition \ref{def:finorb-circle}, for each $s=1,\ldots, d$, and for all $\bar{q}+1\leq k\leq \bar{q}+rp$ we define
$$\epsilon^t_{k,s}:=\begin{cases}e^t_{k+q-\bar{q},s} & \text{ if }\bar{q}+1\leq k\leq \bar{q}+rp-q \\ \si^{k}(F)e^t_{k+q-\bar{q}-rp,s} & \text{ if }\bar{q}+rp-q+1\leq k\leq \bar{q}+rp \end{cases}$$
then, it is not hard to check that for each $s=1,\ldots,d$, the span of $\{\epsilon^t_{k,s}~|~\bar{q}+1\leq k\leq \bar{q}+rp\}$ gives a module isomorphic to $V^t(\om,\bar{q},w')$, with $w'=w_{q+1}\cdots w_{rp} w_1 w_2\cdots w_{q-1}$.
\end{proof}

\section{Tensor products of simple and indecomposable weight modules}\label{sec:definitions}

In this section we present new results regarding the tensor products of the weight modules defined in the previous section.

\begin{Lemma}\label{lemma:tens-FG}Suppose that $B^t_\om=B^{t'}_\om=B^{tt'}_\om=\emptyset$, then
\[ V^t(\om,1^p, F)\otimes_R V^{t'}(\om,1^p,G)\simeq V^{tt'}(\om,1^p, F\otimes G),\]
where $F\in \Aut_{\K}(\K^d)$, $G\in \Aut_{\K}(\K^{d'})$.
\end{Lemma}

\begin{proof}
Consider the basis $\{e^t_{k,s}~|~1\leq k\leq p,~1\leq s\leq d\}$ of $V^t(\om,1^p, F)$ and the basis $\{e^{t'}_{k',s'}~|~1\leq k'\leq p,~1\leq s'\leq d'\}$ of $V^{t'}(\om,1^p, G)$, then we have a basis $\{\epsilon_{k,s,s'}~|~1\leq k\leq p, ~1\leq s\leq d,~1\leq s'\leq d'\}$ for $V^t(\om,1^p, F)\otimes_R V^{t'}(\om,1^p,G)$, where $\epsilon_{k,s,s'}=e^t_{k,s}\otimes e^{t'}_{k,s'}.$

Then, since $w=1^p$ for both modules, if $k\neq p$ we have
\begin{align*}
X\epsilon_{k,s,s'}&= X\cdot(e^t_{k,s}\otimes e^{t'}_{k,s'}) \\
&= Xe^t_{k,s}\otimes X e^{t'}_{k,s'} \\
&= \sigma(t)e^t_{k+1,s}\otimes \si(t') e^{t'}_{k+1,s'} \\
&= \si(t)\si(t')(e^t_{k+1,s}\otimes e^{t'}_{k+1,s'}) \\
&= \si(tt') \epsilon_{k+1,s,s'}.
\end{align*}
If $k= p$, we get
\begin{align*}
X\epsilon_{p,s,s'}&= X\cdot(e^t_{p,s}\otimes e^{t'}_{p,s'}) \\
&= Xe^t_{p,s}\otimes X e^{t'}_{p,s'} \\
&= \sigma(t)\si(F)e^t_{1,s}\otimes\si(t')\si(G)e^{t'}_{1,s'}\\
&= \si(t)\si(t')(\si(F)\otimes\si(G))\cdot ( e^t_{1,s}\otimes e^{t'}_{1,s'})\\
&= \si(tt')\si(F\otimes G) \epsilon_{1,s,s'}.
\end{align*}

We can similarly describe the action of $Y$ on the tensor product, which gives the desired result.
\end{proof}

\begin{Remark}
In general, even if $F=F_f$ and $G=F_g$, $F_f\otimes F_g$ will not be a companion matrix for a polynomial, in particular there is no guarantee of $V^{tt'}(\om,1^p,F\otimes G)$ being simple nor indecomposable regardless of whether $V^t(\om,1^p,F)$ and $V^{t'}(\om,1^p,G)$ are.
\end{Remark}

\begin{Corollary}
If $B^t_\om=\emptyset$, we can decompose
\[ V^t(\om,1^p,f)\simeq V^1(\om,1^p,f)\otimes_R V^t(\om,1^p,x-1).\]
\end{Corollary}

\begin{proof}
This follows immediately from Lemma \ref{lemma:tens-FG} because $F_{x-1}=\begin{bmatrix} 1 \end{bmatrix}=I_1$ is the identity matrix of a one dimensional vector space, hence $F_f\otimes F_{x-1}=F_f$ for all $f\in \K[x]$.
\end{proof}

\begin{Lemma}\label{lem:nobr-tens-2nd}
Let $w=w_1w_2\ldots w_{r p}\in \mathbf{D}^t_\om$, and $F\in\Aut_{\K}(\K^d)$, then
\[V^1(\om,1^p,F)\otimes_R V^t(\om,w,x-1)\simeq V^t(\om,w,F^{r}).\]
\end{Lemma}

\begin{proof}
Consider the basis $\{e^1_{k,s}~|~1\leq k\leq p,~1\leq s\leq d\}$ of $V^1(\om,1^p, F)$ and the basis $\{e^{t}_{k',1}~|~1\leq k'\leq rp \}$ of $V^{t'}(\om,w, x-1)$. 
Since $\sigma^k(F)\in \Aut_{R/\sigma^k(\mf_0)}(\oplus_{s=1}^d (R/\sigma^k(\mf_0))e^1_{k,s})$, we can define a basis for $V^1(\om,1^p,F)\otimes_R V^t(\om,w,x-1)$ as
$$\epsilon_{k+\ell p,s}:=(\sigma^k(F^\ell)\cdot e^1_{k,s})\otimes e^{t}_{k+\ell p,1},\qquad 1\leq k\leq p,~0\leq\ell\leq r-1.$$

The action of $X$ on this basis is given by
\begin{align*}
X \epsilon_{k+\ell p,s}&= X (\sigma^k(F^\ell)\cdot e^1_{k,s})\otimes X e^{t}_{k+\ell p,1} \\
&= \sigma^{k+1}(F^\ell)\cdot X e^1_{k,s}\otimes X e^{t}_{k+\ell p,1}\\
&= \begin{cases}
\sigma^{k+1}(F^\ell)\cdot e^1_{k+1,s}\otimes\si(t) e^t_{k+\ell p+1,s},&\text{if $k\neq p$, $w_k=1$},\\
\sigma^{k+1}(F^\ell)\cdot e^1_{k+1,s}\otimes e^t_{k+\ell p+1,s},&\text{if $k\neq p$, $w_k=x$},\\
\sigma^{p+1}(F^\ell)\cdot \si(F)e^1_{1,s}\otimes\si(t) e^t_{p+\ell p+1,s},&\text{if $k= p$, $\ell\neq r-1$, $w_k=1$},\\
\sigma^{p+1}(F^\ell)\cdot \si(F)e^1_{1,s}\otimes e^t_{p+\ell p+1,s},&\text{if $k= p$, $\ell\neq r-1$, $w_k=x$},\\
\sigma^{p+1}(F^{r-1})\cdot \si(F)e^1_{1,s}\otimes\si(t) e^t_{1,s},&\text{if $k= p$, $\ell= r-1$, $w_k=1$},\\
\sigma^{p+1}(F^{r-1})\cdot \si(F)e^1_{1,s}\otimes e^t_{1,s},&\text{if $k= p$, $\ell= r-1$, $w_k=x$},\\
0,&\text{ if $w_k\in\{y,0\}$};
\end{cases}\\
&= \begin{cases}
\si(t) (\sigma^{k+1}(F^\ell)\cdot e^1_{k+1,s}\otimes e^t_{k+\ell p+1,s}),&\text{if $k\neq p$, $w_k=1$},\\
\sigma^{k+1}(F^\ell)\cdot e^1_{k+1,s}\otimes e^t_{k+\ell p+1,s},&\text{if $k\neq p$, $w_k=x$},\\
\si(t)(\sigma(F^{\ell+1})\cdot e^1_{1,s}\otimes e^t_{1+(\ell+1)p,s}),&\text{if $k= p$, $\ell\neq r-1$, $w_k=1$},\\
\sigma(F^{\ell+1})\cdot e^1_{1,s}\otimes e^t_{1+(\ell+1)p,s},&\text{if $k= p$, $\ell\neq r-1$, $w_k=x$},\\
\si(t)(\sigma(F^{r})\cdot e^1_{1,s}\otimes e^t_{1,s}),&\text{if $k= p$, $\ell= r-1$, $w_k=1$},\\
\sigma(F^{r})\cdot e^1_{1,s}\otimes e^t_{1,s},&\text{if $k= p$, $\ell= r-1$, $w_k=x$},\\
0,&\text{ if $w_k\in\{y,0\}$};
\end{cases}\\
&= \begin{cases}
\si(t) \epsilon_{k+\ell p+1,s},&\text{if $k\neq p$, $w_k=1$},\\
\epsilon_{k+\ell p+1,s},&\text{if $k\neq p$, $w_k=x$},\\
\si(t)\epsilon_{k+\ell p+1,s},&\text{if $k= p$, $\ell\neq r-1$, $w_k=1$},\\
\epsilon_{k+\ell p+1,s},&\text{if $k= p$, $\ell\neq r-1$, $w_k=x$},\\
\si(t)\sigma(F^{r})\cdot \epsilon_{1,s},&\text{if $k= p$, $\ell= r-1$, $w_k=1$},\\
\sigma(F^{r})\cdot \epsilon_{1,s},&\text{if $k= p$, $\ell= r-1$, $w_k=x$},\\
0,&\text{ if $w_k\in\{y,0\}$};
\end{cases}
\end{align*}
where in the fifth and sixth cases we identified $F^r$ with $F^{r}\otimes I_1$ as a linear transformation on $\K^d=\K^d\otimes \K e^{t}_{rp,1}$.
An analogous computation for the action of $Y$ gives the desired result.
\end{proof}

\begin{Definition}
If $f\in\K[x]$, we define a new polynomial $f^{[n]}$ as follows: in a splitting field of $f$, write $f=a\prod_{i=1}^m(x-\alpha_i)$, then
\[f^{[n]}=a^n\prod_{i=1}^m(x-\alpha_i^n). \]
Notice that, since the coefficients of $f^{[n]}$ are symmetric polynomials in the $\alpha_i^n$'s, in particular they are also symmetric polynomials in the $\alpha_i$'s, hence they can be written as polynomials in the elementary symmetric functions of the $\alpha_i$'s which are the coefficients of $f$. It follows that $f^{[n]}\in \K[x]$.
\end{Definition}

\begin{Remark}\label{rem:Fpower}
If $f\in\K[x]$, then $(F_f)^n\sim F_{f^{[n]}}$.
\end{Remark}

\begin{Corollary}\label{cor:tens-f-x-1}
Let $w=w_1\cdots w_{rp}\in \mathbf{D}^t_\om$, and let $f\in\K[x]$, then
\[V^1(\om,1^p,f)\otimes_R V^t(\om,w,x-1)\simeq V^t(\om,w,f^{[r]}).\]
\end{Corollary}

\begin{proof}
This follows immediately from Lemma \ref{lem:nobr-tens-2nd} and Remark \ref{rem:Fpower}.
\end{proof}

\begin{Definition}
Let $u,t, t'\in R$, $w=w_1\cdots w_{rp}\in \mathbf{D}^t_\om$, $w'=w'_1\cdots w'_{r'p}\in \mathbf{D}^{t'}_\om$ such that $(r,r')=d$. Define
\[u^{w,w'}:=\prod_{\substack{1\leq j\leq rr'p/d \\ w_j=1,~w'_j=x}}\si^{p+1-j}(u)\]
where the index $j$ in $w_j$ (resp. $w'_j$) should be interpreted modulo $rp$ (resp. $r'p$). 
\end{Definition}

\begin{Definition}\label{def:monoid}
We equip the set $\{0,1,x,y\}$ with the structure of a monoid by defining a binary commutative operation $\cdot$ such that $z\cdot z=z$, $z\cdot 0=0$, $z\cdot 1=z$, $x\cdot y=0$, for all $z\in \{0,1,x,y\}$.
\end{Definition}

\begin{Definition}\label{def:tensor-word}
Let $w=w_1\cdots w_{rp}\in\mathbf{D}^t_\om$, $w'=w'_1\cdots w'_{r'p}\in\mathbf{D}^{t'}_\om$ with $\gcd(r,r')=d$. 
We define $w\otimes w'\in\mathbf{D}^{tt'}_\om$ as follows
\[w\otimes w'=(w_1\cdot w'_{1})(w_2\cdot w'_{2})\ldots (w_{rr'p/d}\cdot w'_{rr'p/d})\]
where the product $\cdot$ in each parenthesis is given in Definition \ref{def:monoid}. 
The indices for $w_k$ (resp. $w'_k$) in the formula should be interpreted modulo $rp$ (resp. $r'p$).
\end{Definition}
\begin{Definition}\label{def:word-shift}
Let $w=w_1\cdots w_{rp}\in\mathbf{D}^t_\om$, for $j\in\Z$, define $w[j]=z_1\cdots z_{rp}\in\mathbf{D}^t_\om$ where $z_k=w_{k+jp}$ (with the index interpreted modulo $rp$).
\end{Definition}

\begin{Lemma}\label{lem:2nd-kind-tensor}
Let $w=w_1\cdots w_{rp}\in\mathbf{D}^t_\om$, $w'=w'_1\cdots w'_{r'p}\in\mathbf{D}^{t'}_\om$ with $\gcd(r,r')=d$. Then
\[V^t(\om,w,x-1)\otimes_R V^{t'}(\om,w',x-1)\simeq \bigoplus_{j=0}^{d-1}V^{tt'}(\om, w\otimes (w'[j]),x-t^{w,w'[j]}t'^{w'[j],w}).\]
\end{Lemma}

\begin{proof}
Recall the basis elements $\{e^t_{k,1}\}_{k=1}^{rp}$ of $V^t(\om,w,x-1)$ and $\{e^{t'}_{k,1}\}_{k=1}^{r'p}$ of $V^{t'}(\om,w',x-1)$. Let $0\leq j\leq d-1$ and, for all $1\leq k\leq rr'p/d$ define a basis element $\epsilon_k^j$ of $V^t(\om,w,x-1)\otimes_R V^{t'}(\om,w',x-1)$ as follows
$$\epsilon_k^j:=\prod_{\substack{1\leq i\leq k-1 \\ w_i=1,~w'[j]_i=x}}\si^{k-i}(t)\prod_{\substack{1\leq i'\leq k-1 \\ w_{i'}=x,~w'[j]_{i'}=1}}\si^{k-i'}(t')e^t_{\bar{k},1}\otimes e^{t'}_{\bar{k'},1}$$
where $1\leq \bar{k}\leq rp$ is such that $k\equiv \bar{k}\mod {rp}$, and $1\leq \bar{k'}\leq r'p$ is such that $k+jp\equiv \bar{k'}\mod {r'p}$.

We claim that the module generated by $\{\epsilon_k^j\}_{k=1}^{rr'p/d}$ is isomorphic to $V^{tt'}(\om, w\otimes (w'[j]),x-t^{w,w'[j]}t'^{w'[j],w})$.

When $1\leq k<rr'p/d$, the action of $X$ on this basis is given by
\begin{align*}
X \epsilon^{j}_{k}&= X\prod_{\substack{1\leq i\leq k-1 \\ w_i=1,~w'[j]_i=x}}\si^{k-i}(t)\prod_{\substack{1\leq i'\leq k-1 \\ w_{i'}=x,~w'[j]_{i'}=1}}\si^{k-i'}(t')e^t_{\bar{k},1}\otimes e^{t'}_{\bar{k'},1} \\
&=\prod_{\substack{1\leq i\leq k-1 \\ w_i=1,~w'[j]_i=x}}\si^{k+1-i}(t)\prod_{\substack{1\leq i'\leq k-1 \\ w_{i'}=x,~w'[j]_{i'}=1}}\si^{k+1-i'}(t') X e^t_{\bar{k},1}\otimes X e^{t'}_{\bar{k'},1}\\
&= \begin{cases}\displaystyle
\prod_{\substack{1\leq i\leq k-1 \\ w_i=1,~w'[j]_i=x}}\si^{k+1-i}(t)\prod_{\substack{1\leq i'\leq k-1 \\ w_{i'}=x,~w'[j]_{i'}=1}}\si^{k+1-i'}(t') \sigma(t) e^t_{\bar{k+1},1}\otimes \sigma(t') e^{t'}_{\bar{k'+1},1},&\text{if $w_k=w'[j]_k=1$},\\
\displaystyle \prod_{\substack{1\leq i\leq k-1 \\ w_i=1,~w'[j]_i=x}}\si^{k+1-i}(t)\prod_{\substack{1\leq i'\leq k-1 \\ w_{i'}=x,~w'[j]_{i'}=1}}\si^{k+1-i'}(t') \sigma(t) e^t_{\bar{k+1},1}\otimes e^{t'}_{\bar{k'+1},1},&\text{if $w_k=1$, $w'[j]_k=x$},\\
\displaystyle \prod_{\substack{1\leq i\leq k-1 \\ w_i=1,~w'[j]_i=x}}\si^{k+1-i}(t)\prod_{\substack{1\leq i'\leq k-1 \\ w_{i'}=x,~w'[j]_{i'}=1}}\si^{k+1-i'}(t')  e^t_{\bar{k+1},1}\otimes \sigma(t') e^{t'}_{\bar{k'+1},1},&\text{if $w_k=x$, $w'[j]_k=1$},\\
\displaystyle \prod_{\substack{1\leq i\leq k-1 \\ w_i=1,~w'[j]_i=x}}\si^{k+1-i}(t)\prod_{\substack{1\leq i'\leq k-1 \\ w_{i'}=x,~w'[j]_{i'}=1}}\si^{k+1-i'}(t')  e^t_{\bar{k+1},1}\otimes e^{t'}_{\bar{k'+1},1},&\text{if $w_k=w'[j]_k=x$},\\
0,&\text{otherwise};
\end{cases}\\
&= \begin{cases}
\displaystyle \sigma(tt') \prod_{\substack{1\leq i\leq k \\ w_i=1,~w'[j]_i=x}}\si^{k+1-i}(t)\prod_{\substack{1\leq i'\leq k \\ w_{i'}=x,~w'[j]_{i'}=1}}\si^{k+1-i'}(t') e^t_{\bar{k+1},1}\otimes e^{t'}_{\bar{k'+1},1},&\text{if $w_k=w'[j]_k=1$},\\
\displaystyle \prod_{\substack{1\leq i\leq k \\ w_i=1,~w'[j]_i=x}}\si^{k+1-i}(t)\prod_{\substack{1\leq i'\leq k \\ w_{i'}=x,~w'[j]_{i'}=1}}\si^{k+1-i'}(t') e^t_{\bar{k+1},1}\otimes e^{t'}_{\bar{k'+1},1},&\text{if $w_k=1$, $w'[j]_k=x$},\\
\displaystyle \prod_{\substack{1\leq i\leq k \\ w_i=1,~w'[j]_i=x}}\si^{k+1-i}(t)\prod_{\substack{1\leq i'\leq k \\ w_{i'}=x,~w'[j]_{i'}=1}}\si^{k+1-i'}(t')  e^t_{\bar{k+1},1}\otimes e^{t'}_{\bar{k'+1},1},&\text{if $w_k=x$, $w'[j]_k=1$},\\
\displaystyle \prod_{\substack{1\leq i\leq k \\ w_i=1,~w'[j]_i=x}}\si^{k+1-i}(t)\prod_{\substack{1\leq i'\leq k \\ w_{i'}=x,~w'[j]_{i'}=1}}\si^{k+1-i'}(t')  e^t_{\bar{k+1},1}\otimes e^{t'}_{\bar{k'+1},1},&\text{if $w_k=w'[j]_k=x$},\\
0,&\text{otherwise};
\end{cases}\\
&= \begin{cases}
\sigma(tt') \epsilon^j_{k+1} &\text{if $(w\otimes w'[j])_k=1$} \\
\epsilon^j_{k+1} &\text{if $(w\otimes w'[j])_k=x$} \\
0,&\text{otherwise}.
\end{cases}
\end{align*}

When $k=rr'p/d$, we have

\begin{align*}
X \epsilon^{j}_{rr'p/d}&= X\prod_{\substack{1\leq i\leq rr'p/d-1 \\ w_i=1,~w'[j]_i=x}}\si^{rr'p/d-i}(t)\prod_{\substack{1\leq i'\leq rr'p/d-1 \\ w_{i'}=x,~w'[j]_{i'}=1}}\si^{rr'p/d-i'}(t')e^t_{rp,1}\otimes e^{t'}_{r'p,1} \\
&=\prod_{\substack{1\leq i\leq rr'p/d-1 \\ w_i=1,~w'[j]_i=x}}\si^{p+1-i}(t)\prod_{\substack{1\leq i'\leq rr'p/d-1 \\ w_{i'}=x,~w'[j]_{i'}=1}}\si^{p+1-i'}(t') X e^t_{rp,1}\otimes X e^{t'}_{r'p,1}\\
&= \begin{cases}
\displaystyle \prod_{\substack{1\leq i\leq rr'p/d-1 \\ w_i=1,~w'[j]_i=x}}\si^{p+1-i}(t)\prod_{\substack{1\leq i'\leq rr'p/d-1 \\ w_{i'}=x,~w'[j]_{i'}=1}}\si^{p+1-i'}(t') \sigma(t) e^t_{1,1}\otimes \sigma(t') e^{t'}_{1,1},&\text{if $w_{rp}=w'[j]_{r'p}=1$},\\
\displaystyle \prod_{\substack{1\leq i\leq rr'p/d-1 \\ w_i=1,~w'[j]_i=x}}\si^{p+1-i}(t)\prod_{\substack{1\leq i'\leq rr'p/d-1 \\ w_{i'}=x,~w'[j]_{i'}=1}}\si^{p+1-i'}(t') \sigma(t) e^t_{1,1}\otimes e^{t'}_{1,1},&\text{if $w_{rp}=1$, $w'[j]_{r'p}=x$},\\
\displaystyle \prod_{\substack{1\leq i\leq rr'p/d-1 \\ w_i=1,~w'[j]_i=x}}\si^{p+1-i}(t)\prod_{\substack{1\leq i'\leq rr'p/d-1 \\ w_{i'}=x,~w'[j]_{i'}=1}}\si^{p+1-i'}(t')  e^t_{1,1}\otimes \sigma(t') e^{t'}_{1,1},&\text{if $w_{rp}=x$, $w'[j]_{r'p}=1$},\\
\displaystyle \prod_{\substack{1\leq i\leq rr'p/d-1 \\ w_i=1,~w'[j]_i=x}}\si^{p+1-i}(t)\prod_{\substack{1\leq i'\leq rr'p/d-1 \\ w_{i'}=x,~w'[j]_{i'}=1}}\si^{p+1-i'}(t')  e^t_{1,1}\otimes e^{t'}_{1,1},&\text{if $w_{rp}=w'[j]_{r'p}=x$},\\
0,&\text{otherwise};
\end{cases} \\
&= \begin{cases}
\displaystyle \sigma(tt') \prod_{\substack{1\leq i\leq rr'p/d \\ w_i=1,~w'[j]_i=x}}\si^{p+1-i}(t)\prod_{\substack{1\leq i'\leq rr'p/d \\ w_{i'}=x,~w'[j]_{i'}=1}}\si^{p+1-i'}(t') e^t_{1,1}\otimes e^{t'}_{1,1} & \text{if $(w\otimes w'[j])_{rr'p/d}=1$}, \\
\displaystyle \prod_{\substack{1\leq i\leq rr'p/d \\ w_i=1,~w'[j]_i=x}}\si^{p+1-i}(t)\prod_{\substack{1\leq i'\leq rr'p/d \\ w_{i'}=x,~w'[j]_{i'}=1}}\si^{p+1-i'}(t') e^t_{1,1}\otimes e^{t'}_{1,1} & \text{if $(w\otimes w'[j])_{rr'p/d}=x$},\\
0,&\text{otherwise};
\end{cases} \\
&= \begin{cases}
 \sigma(tt') t^{w,w'[j]}t'^{w'[j],w} \epsilon^j_1 & \text{if $(w\otimes w'[j])_{rr'p/d}=1$}, \\
 t^{w,w'[j]}t'^{w'[j],w} \epsilon^j_1 & \text{if $(w\otimes w'[j])_{rr'p/d}=x$}, \\
 0,&\text{otherwise}.
\end{cases}
\end{align*}
An analogous computation for the action of $Y$ proves the claim. By examining the indices involved in the basis elements, it is clear that the spaces spanned by $\{\epsilon_k^j\}_{k=1}^{rr'p/d}$ and $\{\epsilon_k^{j'}\}_{k=1}^{rr'p/d}$ with $0\leq j<j'\leq d-1$ intersect trivially, which gives the direct sum in the statement of the Lemma. Finally, by a dimension count we can conclude that the direct sum does give the whole tensor product, which finishes the proof.
\end{proof}

\begin{Lemma} \label{lem:1st-tensor-1st}
Let $0\leq i\leq i'\leq p-1$, $w=w_{i+1}\cdots w_{i+\ell}\in\mathbf{D}^{t}_\om$, $w'=w'_{i'+1}\cdots w'_{i'+\ell'}\in\mathbf{D}^{t'}_\om$, let $V=V^t(\om,i,w)$, $V'=V^{t'}(\om,i',w')$. Suppose that 
$\dim_{R/\si^{i'+1}(\mf_0)}(V_{\si^{i'+1}(\mf_0)})=c$ and $\dim_{R/\si^{i+1}(\mf_0)}(V'_{\si^{i+1}(\mf_0)})=c'$, then
\[ V\otimes_R V'=\bigoplus_{j=0}^{c-1} V^{tt'}(\om,i',(w')^{(j)})\oplus \bigoplus_{j=1}^{c'-\delta_{i,i'}} V^{tt'}(\om,i,w^{(j)})\]
where \[(w')^{(j)}=(w_{i'+1+jp}\cdot w'_{i'+1})( w_{i'+2+jp}\cdot w'_{i'+2})\ldots ( w_{i'+jp+\min\{i+\ell-i'-jp,\ell'\}}\cdot w'_{i'+\min\{i+\ell-i'-jp,\ell'\}});\]
\[w^{(j)}=(w_{i+1}\cdot w'_{i+1+jp})( w_{i+2}\cdot w'_{i+2+jp})\ldots (w_{i+\min\{\ell, i'+\ell'-i-jp\}}\cdot w'_{i+jp+\min\{\ell, i'+\ell'-i-jp\}}).\]

Here we adopt the convention that if $i+\ell=i'+jp$, (resp. $i'+\ell'=i+jp$) then $(w')^{(j)}=\emptyset$ (resp. $w^{(j)}=\emptyset$).
\end{Lemma}

\begin{proof}
Recall the basis elements $\{e^t_{k}\}_{k=i+1}^{i+\ell+1}$ of $V^t(\om,i,w)$ and $\{e^{t'}_{k}\}_{k=1}^{i'+\ell+1}$ of $V^{t'}(\om,i,w')$. Let $0\leq j\leq c-1$ and, for all $i'+1\leq k\leq i'+\min\{i+\ell-i'-jp,\ell'\}+1$ define a basis element $\epsilon_k^j$ of $V\otimes_R V'$ as follows
\[\epsilon_k^j:=\prod_{\substack{i'+1\leq m\leq k-1 \\ w_{m+jp}=1,~w'_m=x}}\si^{k-m}(t)\prod_{\substack{i'+1\leq n\leq k-1 \\ w_{n+jp}=x,~w'_{n}=1}}\si^{k-n}(t')e^t_{k+jp}\otimes e^{t'}_{k}.\]

We claim that the module generated by $\{\epsilon_k^j~|~i'+1\leq k\leq i'+\min\{i+\ell-i'-jp,\ell'\}+1\}$ is isomorphic to $V^{tt'}(\om,i',(w')^{(j)})$.

First of all, if $k=i'+\min\{i+\ell-i'-jp,\ell'\}+1$ then

\begin{align*}X \epsilon^{j}_{k}&= X\prod_{\substack{i'+1\leq m\leq k-1 \\ w_{m+jp}=1,~w'_m=x}}\si^{k-m}(t)\prod_{\substack{i'+1\leq n\leq k-1 \\ w_{n+jp}=x,~w'_{n}=1}}\si^{k-n}(t')e^t_{k+jp}\otimes e^{t'}_{k} \\
&=\prod_{\substack{i'+1\leq m\leq k-1 \\ w_{m+jp}=1,~w'_m=x}}\si^{k-m}(t)\prod_{\substack{i'+1\leq n\leq k-1 \\ w_{n+jp}=x,~w'_{n}=1}}\si^{k-n}(t')Xe^t_{k+jp}\otimes Xe^{t'}_{k} \\
& =0.
\end{align*}

For $i'+1\leq k< i'+\min\{i+\ell-i'-jp,\ell'\}+1$ (such an index can only exist if $i+\ell\neq i'+jp$, which means that $(w')^{(j)}\neq\emptyset$) we have:
\begin{align*}
X \epsilon^{j}_{k}&= X\prod_{\substack{i'+1\leq m\leq k-1 \\ w_{m+jp}=1,~w'_m=x}}\si^{k-m}(t)\prod_{\substack{i'+1\leq n\leq k-1 \\ w_{n+jp}=x,~w'_{n}=1}}\si^{k-n}(t')e^t_{k+jp}\otimes e^{t'}_{k} \\
&=\prod_{\substack{i'+1\leq m\leq k-1 \\ w_{m+jp}=1,~w'_m=x}}\si^{k-m}(t)\prod_{\substack{i'+1\leq n\leq k-1 \\ w_{n+jp}=x,~w'_{n}=1}}\si^{k-n}(t')Xe^t_{k+jp}\otimes Xe^{t'}_{k}\\
&= \begin{cases}\displaystyle
\prod_{\substack{i'+1\leq m\leq k-1 \\ w_{m+jp}=1,~w'_m=x}}\si^{k-m}(t)\prod_{\substack{i'+1\leq n\leq k-1 \\ w_{n+jp}=x,~w'_{n}=1}}\si^{k-n}(t')\sigma(t)e^t_{k+jp+1}\otimes \sigma(t')e^{t'}_{k+1},&\text{if $w_{k+jp}=w'_k=1$},\\
\displaystyle \prod_{\substack{i'+1\leq m\leq k-1 \\ w_{m+jp}=1,~w'_m=x}}\si^{k-m}(t)\prod_{\substack{i'+1\leq n\leq k-1 \\ w_{n+jp}=x,~w'_{n}=1}}\si^{k-n}(t')\sigma(t)e^t_{k+jp+1}\otimes e^{t'}_{k+1},&\text{if $w_{k+jp}=1$, $w'_k=x$},\\
\displaystyle \prod_{\substack{i'+1\leq m\leq k-1 \\ w_{m+jp}=1,~w'_m=x}}\si^{k-m}(t)\prod_{\substack{i'+1\leq n\leq k-1 \\ w_{n+jp}=x,~w'_{n}=1}}\si^{k-n}(t') e^t_{k+jp+1}\otimes \sigma(t')e^{t'}_{k+1},&\text{if $w_{k+jp}=x$, $w'_k=1$},\\
\displaystyle \prod_{\substack{i'+1\leq m\leq k-1 \\ w_{m+jp}=1,~w'_m=x}}\si^{k-m}(t)\prod_{\substack{i'+1\leq n\leq k-1 \\ w_{n+jp}=x,~w'_{n}=1}}\si^{k-n}(t') e^t_{k+jp+1}\otimes e^{t'}_{k+1},&\text{if $w_{k+jp}=w'_k=x$},\\
0,&\text{otherwise};
\end{cases} \\
&= \begin{cases}
\displaystyle \sigma(tt') \prod_{\substack{i'+1\leq m\leq k \\ w_{m+jp}=1,~w'_m=x}}\si^{k-m}(t)\prod_{\substack{i'+1\leq n\leq k \\ w_{n+jp}=x,~w'_{n}=1}}\si^{k-n}(t')e^t_{k+jp+1}\otimes e^{t'}_{k+1},&\text{if $w_{k+jp}=w'_k=1$},\\
\displaystyle \prod_{\substack{i'+1\leq m\leq k \\ w_{m+jp}=1,~w'_m=x}}\si^{k-m}(t)\prod_{\substack{i'+1\leq n\leq k \\ w_{n+jp}=x,~w'_{n}=1}}\si^{k-n}(t')\sigma(t)e^t_{k+jp+1}\otimes e^{t'}_{k+1},&\text{if $w_{k+jp}=1$, $w'_k=x$},\\
\displaystyle \prod_{\substack{i'+1\leq m\leq k \\ w_{m+jp}=1,~w'_m=x}}\si^{k-m}(t)\prod_{\substack{i'+1\leq n\leq k \\ w_{n+jp}=x,~w'_{n}=1}}\si^{k-n}(t') e^t_{k+jp+1}\otimes \sigma(t')e^{t'}_{k+1},&\text{if $w_{k+jp}=x$, $w'_k=1$},\\
\displaystyle \prod_{\substack{i'+1\leq m\leq k \\ w_{m+jp}=1,~w'_m=x}}\si^{k-m}(t)\prod_{\substack{i'+1\leq n\leq k \\ w_{n+jp}=x,~w'_{n}=1}}\si^{k-n}(t') e^t_{k+jp+1}\otimes e^{t'}_{k+1},&\text{if $w_{k+jp}=w'_k=x$},\\
0,&\text{otherwise};
\end{cases}\\
&= \begin{cases}
\sigma(tt') \epsilon^j_{k+1} &\text{if $(w')^{(j)}_k=1$} \\
\epsilon^j_{k+1} &\text{if $(w')^{(j)}_k=x$} \\
0,&\text{otherwise}.
\end{cases}
\end{align*}
An analogous computation for the action of $Y$ proves the claim. 

We also define, for all $1\leq j\leq c'-\delta_{i,i'}$, and $i+1\leq k\leq i+\min\{\ell,i'+\ell'-i-jp\}+1$ a basis element $(\epsilon')_k^j$ of $V\otimes_R V'$ as follows
$$(\epsilon')_k^j:=\prod_{\substack{i'+1\leq m\leq k-1 \\ w_{m}=1,~w'_{m+jp}=x}}\si^{k-m}(t)\prod_{\substack{i'+1\leq n\leq k-1 \\ w_{n}=x,~w'_{n+jp}=1}}\si^{k-n}(t')e^t_{k}\otimes e^{t'}_{k+jp}.$$

With a computation of the action of $X$ and $Y$ analogous to the previous one, it can also be checked that the module generated by $\{(\epsilon')_k^j~|~i+1\leq k\leq i+\min\{\ell,i'+\ell'-i-jp\}+1\}$ is isomorphic to $V^{tt'}(\om,i,w^{(j)})$.

Finally, the result follows by examining the indices to check that it is a direct sum because the modules we defined intersect trivially, and by a dimension count that shows the isomorphism.
\end{proof}

\begin{Example}\label{ex:tensor}
$p=3$, $\om=\{\Fm_0, \Fm_1, \Fm_2\}$, $B_\om^t = \{\Fm_0,\Fm_2\}$, $B_\om^{t'}=\{\Fm_1,\Fm_2\}$. Let $i=2$, $w=x1x$, $i'=2$, $w'=1yx10x1$. Then we can compute $(w')^{(j)}$ and $w^{(j)}$ as follows: we line up the shifted words vertically and multiply each pair of vertically aligned letters according to the rules in Definition \ref{def:monoid}.

$$\left( \begin{tabular}{l} $x1x$ \\ $1yx10x1$ \\ \hline $xyx$\end{tabular}\right)\Rightarrow (w')^{(0)}=xyx, \qquad
\left( \begin{tabular}{l} $x1x$ \\ $\phantom{x1x}1yx10x1$ \\ \hline $\phantom{\emptyset}$    \end{tabular} \right) \Rightarrow (w')^{(1)}=\emptyset $$
$$ \left( \begin{tabular}{l} $\phantom{x1x}x1x$ \\ $1yx10x1$ \\ \hline $\phantom{x1x}x0x$\end{tabular}\right) \Rightarrow w^{(1)}=x0x, \qquad 
\left(\begin{tabular}{l} $\phantom{x1xx1x}x1x$ \\ $1yx10x1$ \\ \hline $\phantom{x1xx1x}x$   \end{tabular}\right) \Rightarrow w^{(2)}=x $$
Hence, by Lemma \ref{lem:1st-tensor-1st}, we have the following:
\begin{align*}V^t(\omega, 2, x1x) \otimes V^{t'}(\omega, 2,1yx10x1) &\simeq V^{tt'}(\om,2,xyx)\oplus V^{tt'}(\om,2,\emptyset)\oplus V^{tt'}(\om,2,x0x)\oplus V^{tt'}(\om,2,x) \\
&\simeq V^{tt'}(\om,2,xyx)\oplus V^{tt'}(\om,2,\emptyset)\oplus \left( V^{tt'}(\om,2,x)\oplus V^{tt'}(\om,1,x) \right) \\
& \phantom{\oplus\oplus}\oplus V^{tt'}(\om,2,x)  \\
&=V^{tt'}(\om,2,xyx)\oplus V^{tt'}(\om,2,\emptyset)\oplus V^{tt'}(\om,2,x)^{\oplus 2}\oplus V^{tt'}(\om,1,x). 
\end{align*}
Note that the second factor of the tensor product is decomposable, by Remark \ref{rem:decomp-type1},
\[V^{t'}(\omega, 2,1yx10x1)\simeq V^{t'}(\omega,2,1yx1)\oplus V^{t'}(\omega,1,x1)\]
so we could also have done the computation using
\begin{align*}V^t(\omega, 2, x1x) \otimes V^{t'}(\omega, 2,1yx10x1) &\simeq 
V^t(\omega, 2, x1x) \otimes \left(V^{t'}(\omega,2,1yx1)\oplus V^{t'}(\omega,1,x1) \right) \\
&\simeq \left(V^t(\omega, 2, x1x) \otimes V^{t'}(\omega,2,1yx1)\right) \oplus \left(V^t(\omega, 2, x1x) \otimes V^{t'}(\omega,1,x1)\right),
\end{align*}
and then computed each of those tensor products using the same method as above.
\end{Example}

\begin{Example}
Figure \ref{fig:tensor-product-example} gives a visual representation of the calculation in Example \ref{ex:tensor}.
\end{Example}

\begin{figure}
\begin{align*}
&
\begin{tikzpicture}[baseline=-.5ex, vertex/.style={draw,circle,fill=white,minimum size=5pt, inner sep=0pt}]
\node[vertex] (A) at (1,0) [label=below:\small $e^t_3$] {};
\node[vertex] (B) at ({cos(120)},{sin(120)}) [label=below:\small $e^t_4$] {};
\node[vertex] (C) at ({cos(240)},{sin(240)}) [label=below:\small $e^t_5$] {};
\node[vertex] (D) at (1.4,0) [label=above:\small $e^t_6$] {};
\draw[->] (A) edge[out=90, in=30] (B);
\draw[<->] (B) edge[out=180+30, in=150] (C);
\draw[->] (C) edge[out=180+150, in=180+90] (D);
\end{tikzpicture}
\otimes
\begin{tikzpicture}[baseline=-.5ex, vertex/.style={draw,circle,fill=white,minimum size=5pt, inner sep=0pt}]
\node[vertex] (A) at (1,0) [label=left:\small $e^{t'}_3$] {};
\node[vertex] (B) at ({cos(120)},{sin(120)}) [label=below:\small $e^{t'}_4$] {};
\node[vertex] (C) at ({cos(240)},{sin(240)}) [label=above:\small $e^{t'}_5$] {};
\node[vertex] (D) at (1.4,0) [label=below:\small $e^{t'}_6$] {};
\node[vertex] (E) at ({1.4*cos(120)},{1.4*sin(120)}) [label=left:\small $e^{t'}_7$] {};
\node[vertex] (F) at ({1.4*cos(240)},{1.4*sin(240)}) [label=below:\small $e^{t'}_8$] {};
\node[vertex] (G) at (1.8,0) [label=right:\small $e^{t'}_9$] {};
\node[vertex] (H) at ({1.8*cos(120)},{1.8*sin(120)}) [label=left:\small $e^{t'}_{10}$] {};
\draw[<->] (A) edge[out=90, in=30] (B);
\draw[<-] (B) edge[out=210, in=150] (C);
\draw[->] (C) edge[out=330, in=270] (D);
\draw[<->] (D) edge[out=90, in=30] (E);
\draw[->] (F) edge[out=330, in=270] (G);
\draw[<->] (G) edge[out=90, in=30] (H);
\end{tikzpicture}
\\&=
\begin{tikzpicture}[baseline=-.5ex, vertex/.style={draw,circle,fill=white,minimum size=5pt, inner sep=0pt}]
\node[vertex] (A) at (1,0) [label=left:\small $e^t_3\otimes e^{t'}_3$] {};
\node[vertex] (B) at ({cos(120)},{sin(120)}) [label=above:\small $e^t_4\otimes e^{t'}_4$] {};
\node[vertex] (C) at ({cos(240)},{sin(240)}) [label=below:\small $e^t_5\otimes e^{t'}_5$] {};
\node[vertex] (D) at (1.4,0) [label=above:\small $e^t_6\otimes e^{t'}_6$] {};
\draw[->] (A) edge[out=90, in=30] (B);
\draw[<-] (B) edge[out=210, in=150] (C);
\draw[->] (C) edge[out=330, in=270] (D);
\end{tikzpicture}
\oplus
\begin{tikzpicture}[baseline=-.5ex, vertex/.style={draw,circle,fill=white,minimum size=5pt, inner sep=0pt}]
\node[vertex] (A) at (1,0) [label=below:\small $e^t_3\otimes e^{t'}_6$] {};
\node[vertex] (B) at ({cos(120)},{sin(120)}) [label=below:\small $e^t_4\otimes e^{t'}_7$] {};
\draw[->] (A) edge[out=90, in=30] (B);
\end{tikzpicture}
\oplus
\begin{tikzpicture}[baseline=-.5ex, vertex/.style={draw,circle,fill=white,minimum size=5pt, inner sep=0pt}]
\node[vertex] (A) at (1,0) [label=below:\small $e^t_3\otimes e^{t'}_9$] {};
\node[vertex] (B) at ({cos(120)},{sin(120)}) [label=below:\small $e^t_4\otimes e^{t'}_{10}$] {};
\draw[->] (A) edge[out=90, in=30] (B);
\end{tikzpicture}
\oplus
\begin{tikzpicture}[baseline=-.5ex, vertex/.style={draw,circle,fill=white,minimum size=5pt, inner sep=0pt}]
\node[vertex] (A) at (1,0) [label=below:\small $e^t_6\otimes e^{t'}_3$] {};
\end{tikzpicture}
\oplus
\begin{tikzpicture}[baseline=-.5ex, vertex/.style={draw,circle,fill=white,minimum size=5pt, inner sep=0pt}]
\node[vertex] (A) at ({cos(240)},{sin(240)}) [label=above:\small $e^t_5\otimes e^{t'}_8$] {};
\node[vertex] (B) at (1.4,0) [label=above:\small $e^t_6\otimes e^{t'}_9$] {};
\draw[->] (A) edge[out=330, in=270] (B);
\end{tikzpicture}
\end{align*}
\caption{$V^t(\omega, 2, x1x) \otimes V^{t'}(\omega, 2,1yx10x1) \simeq V^{tt'}(\omega,2 ,xyx)\oplus V^{tt'}(\omega,2,x)^{\oplus 2}\oplus V^{tt'}(\omega,2,\emptyset)\oplus V^{tt'}(\omega,1,x)$ \newline
This computation can be done directly from the picture, as we now explain. We pick as starting point a basis elements that has no edge going counterclockwise to it, like $e^t_3$, and we match it with a basis element that has the same weight in the other tensor factor, like $e^{t'}_6$. Then we simultaneously trace the connected edges moving counterclockwise, with the rule (coming from Def. \ref{def:monoid}) that two arrows pointing in the same direction give the same arrow, two arrows pointing in opposite directions make the edge disappear, and a two headed arrow leaves the other arrow as result. We continue the chain until one of the two basis elements has no counterclockwise edge, so if starting as discussed above we would stop at $e^{t'}_7$ and get the second summand in the result. Then we repeat the process for all starting points and possible matches.} \label{fig:tensor-product-example}
\end{figure}
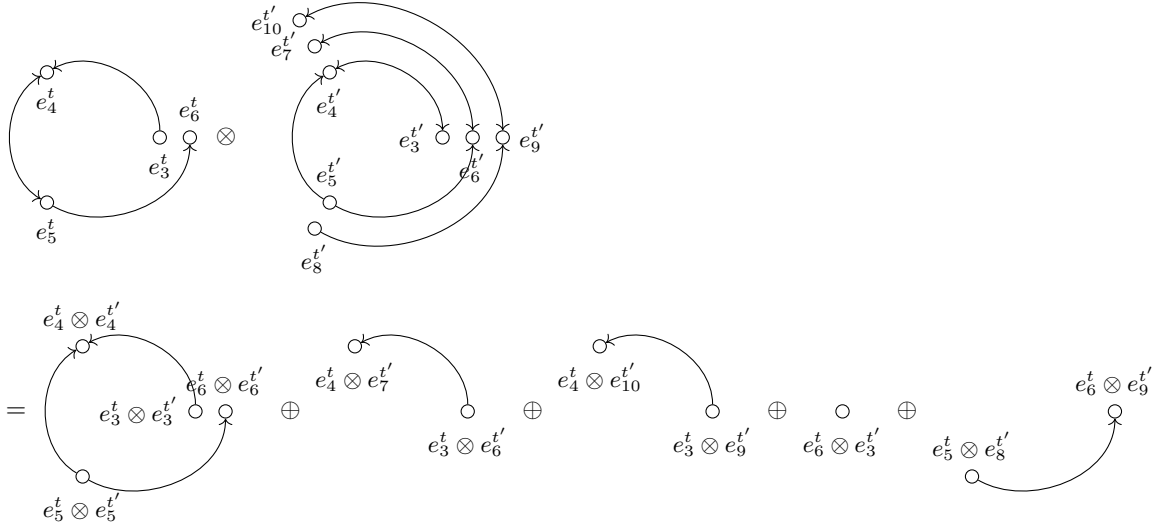

\begin{Lemma}\label{lem:1st-tensor-2nd}
Let $w=w_{i+1}\ldots w_{i+\ell}\in\mathbf{D}_\om^t$, $w'=w'_1\cdots w'_{rp}\in\mathbf{D}^{t'}_\om$,  then
\[ V^{t}(\om,i,w)\otimes_R V^{t'}(\om,w',x-1)=\bigoplus_{j=0}^{r-1}V^{tt'}(\om,i,w^{[j]})\]
where $w^{[j]}=(w_{i+1}\cdot w'_{i+1+jp})\ldots (w_{i+\ell}\cdot w'_{i+\ell+jp})$ and the indices for $w'$ are to be intended modulo $rp$.
\end{Lemma}

\begin{proof}
Recall the basis elements $\{e^t_{k}\}_{k=i+1}^{i+\ell+1}$ of $V^t(\om,i,w)$ and $\{e^{t'}_{k,1}\}_{k=1}^{rp}$ of $V^{t'}(\om,w',x-1)$. Let $0\leq j\leq r-1$ and, for all $i+1\leq k\leq i+\ell+1$ define a basis element $\epsilon_k^j$ of $V^{t}(\om,i,w)\otimes_R V^{t'}(\om,w',x-1)$ as follows
\[\epsilon_k^j:=\prod_{\substack{i+1\leq m\leq k-1 \\ w_{m}=1,~w'_{m+jp}=x}}\si^{k-m}(t)\prod_{\substack{i+1\leq n\leq k-1 \\ w_{n}=x,~w'_{n+jp}=1}}\si^{k-n}(t')e^t_{k}\otimes e^{t'}_{k+jp,1}.\]
We claim that the module generated by $\{\epsilon_k^j~|~i+1\leq k\leq i+\ell+1\}$ is isomorphic to $V^{tt'}(\om,i,w^{[j]})$. The computation to verify this is completely analogous to the ones in Lemmas \ref{lem:2nd-kind-tensor} and  \ref{lem:1st-tensor-1st} so it will be omitted.

Finally, the result follows by examining the indices to check that it is a direct sum because the modules we defined intersect trivially, and by a dimension count that shows the isomorphism.
\end{proof}

\begin{Lemma}\label{lem:1st-tensor-nob}
Let $w=w_{i+1}\ldots w_{i+\ell}\in\mathbf{D}^t_\om$ and suppose $B^{t'}_\om=\emptyset$, and $F\in \Aut_{\K_\om}(\K_\om^d)$, then
\[ V^{t}(\om,i,w)\otimes_R V^{t'}(\om,1^p,F)=V^{tt'}(\om,i,w)^{\oplus d}\]
\end{Lemma}

\begin{proof}
Consider the basis elements $\{e^t_{k}\}_{k=i+1}^{i+\ell+1}$ of $V^t(\om,i,w)$ and $\{e^{t'}_{k,s}\}_{1\leq k\leq p,~1\leq s\leq d}$ of $V^{t'}(\om,1^p,F)$. Define a basis element $\epsilon_k^s$ of $V^{t}(\om,i,w)\otimes_R V^{t'}(\om,1^p,F)$ as follows
$$\epsilon_k^s:=\prod_{\substack{i+1\leq n\leq k-1 \\ w_{n}=x}}\si^{k-n}(t')e^t_{k}\otimes\sigma^k\left(F^{\left\lfloor \frac{k-1}{p}\right\rfloor}\right)e^{t'}_{k,s}.$$

Then, again with a computation analogous to the ones in Lemmas \ref{lem:2nd-kind-tensor} and  \ref{lem:1st-tensor-1st}, it is easy to verify that for all $s=1,\ldots,d$, the module generated by $\{\epsilon_k^s~|~i+1\leq k\leq i+\ell+1\}$ is isomorphic to $V^{tt'}(\om,i,w)$ and to use a dimension count to conclude.
\end{proof}

\section{Presentation of the Grothendieck ring}\label{sec:Groth-ring}

We now apply the results of the previous section to describe the Grothendieck ring $$\mathscr{A}(\omega,\mathsf{M})=\bigoplus_{t\in\mathsf{M}} \C\otimes_{\Z}K_0(A(t)\text{-wmod}_\omega),$$ as defined in Proposition \ref{prop:direct-sum}, in the case where $R=\C[z,z^{-1}]$, $\si(z)=q^{-1}z$ with $q=e^{\frac{2\pi i}{p}}$ a primitive $p$-th root of unity, $p\in\Z_{>0}$. Although some things could be stated in more generality, this is a good model for the finite orbit case, like polynomials with shifts was a good model for the infinite orbit case in \cite[\S 4]{HR}. We choose $\mf_0 = (z-1)$, so that we have the orbit $\om=\{(z-q^k)\}_{k=0}^{p-1}$, with $\Fm_i=\si^i(\mf_0)=(z-q^i)$.

We let $\mathsf{M}$ be the multiplicative submonoid of $R$ generated by $\{z-q^k\mid k=0,1,\ldots,p-1\}$. If $t\in\mathsf{M}$, then $t=(z-1)^{k_0}(z-q)^{k_1}\cdots (z-q^{p-1})^{k_{p-1}}$ for some non-negative integers $k_i$.

We introduce the following elements of the Grothendieck ring $\mathscr{A}(\omega,\mathsf{M})$.

\begin{Definition}\label{def:generators}
\begin{subequations} \label{eq:generators}
\begin{align}
u_\xi &=[V^1(\om,1^p,x-\xi)],\quad \xi\in\C^\times,\\
x_i &= [V^{(z-q^i)}(\om,i,1^{p-1})], \quad 0\le i\le p-1,\\
x_{ij} &= [V^{(z-q^i)(z-q^j)}(\om,i,1^{j-i})], \quad 0\le i<j\le p-1, \\
x_{ji} &= [V^{(z-q^i)(z-q^j)}(\om,j,1^{i-j+p})], \quad 0\le i<j\le p-1, \\
y_i &= [V^{(z-q^i)}(\om,1^{i-1}x1^{p-i},x-1)], \quad 0\le i\le p-1,\\
y_i^\ast &= [V^{(z-q^i)}(\om,1^{i-1}y1^{p-i},x-1)], \quad 0\le i\le p-1.
\end{align}
\end{subequations}
\end{Definition}

\begin{Remark}
We will use the cyclic order on the set $0,1,\ldots,p-1$ given by $0<1<\ldots<p-1<0$. Note that this is a ternary relation, so $i\le j$ is meaningless but $i\le j\le k$ makes sense. By a chain of inequalities we mean it is valid for any consecutive subchain of length three.
\end{Remark}

\begin{Theorem}\label{thm:groth-ring}
The Grothendieck ring $\mathscr{A}(\omega,\mathsf{M})$ is isomorphic to the commutative $\C$-algebra with generators $u_\xi, \xi\in\C^\times$, $x_i, y_i, y_i^\ast$, $0\le i\le p-1$, and $x_{ij}, x_{ji}$, $0\le i<j\le p-1$, subject to the relations

\begin{enumerate}[{\rm (i)}]
\item $u_1=1$,
\item $u_\xi u_{\xi'}=u_{\xi\xi'}$,
\item $u_\xi x_i=x_i$,
\item $u_\xi x_{ij}=x_{ij}$,
\item $y_iy_j^\ast =x_ix_j= x_i y_j=x_iy_j^\ast=x_{ij}+x_{ji}$ when $i\neq j$,
\item $y_iy_i^\ast=x_i y_i = x_iy_i^\ast = x_i^2$,
\item $x_{ij}x_k=x_{ik}x_j+x_{kj}x_i$ if $i<k<j$ in the cyclic order,
\item $x_{ij}x_{k\ell}=0$ if $i<j\leq k<\ell\le i$ in the cyclic order,
\item $x_{ij}x_{k\ell}=x_{kj}x_ix_\ell$ if $i\leq k<j \leq \ell\le i$ in the cyclic order,
\item $x_{ij}x_{k\ell}=x_{k \ell}x_i x_j$ if $i\leq k<\ell\leq j\le i$ in the cyclic order,
\item $x_{ij}x_{k \ell}=x_{i\ell}x_kx_j+x_{kj}x_ix_\ell$ if $i<\ell<k<j\le i$ in the cyclic order,
\item $x_{ij}y_k=x_{ij}y_k^*=x_{ij}x_k$ for all $i\neq j$ and for all $k$ (in particular still true when $k=i,j$).
\end{enumerate}
In particular, the following set of elements is a basis for $\mathscr{A}(\omega,\mathsf{M})$ over $\C$:
\begin{enumerate}[{\rm (1)}]
\item\label{it:1t} $x_{ij}x_{k_1}^{a_1}\cdots x_{k_r}^{a_r}$, with $a_\ell\geq0$ and $i\neq j$ and $j\leq k_\ell\leq i$ in the cyclic order for all $\ell$;
\item\label{it:2t} $x_i^{a}$, $a>0$;
\item\label{it:3t} $u_\xi$, $\xi\in\C^\times$;
\item\label{it:4t} $u_\xi y_{k_1}^{a_1}\cdots y_{k_r}^{a_r}$, with $\xi\in\C^\times$, $a_\ell>0$ for all $\ell$;
\item\label{it:5t} $u_\xi (y^*_{k_1})^{a_1}\cdots (y^*_{k_r})^{a_r}$, with $\xi\in\C^\times$, $a_\ell>0$ for all $\ell$.
\end{enumerate}
\end{Theorem}

\begin{proof}
First we prove that the elements from Definition \ref{def:generators} generate the Grothendieck ring. A basis for the Grothendieck ring is given by simple weight modules, which, given the classification of \cite{DGO}, are one of the following (see Section \ref{sec:definitions} for notation): 
\begin{itemize}
\item $V^t(\om,w,f)$, with $f\in \C[x]$ irreducible polynomial such that $f(0)\neq 0$, and $w=w_1\cdots w_p\in \mathbf{D}^t_{\om}$ such that either $w_i\in\{1,x\}$ for all $i$ or $w_i\in\{1,y\}$ for all $i$;
\item $V^t(\om,i,w)$, with $w=1^k$ for some $0\leq k\leq p-1$ and $0\le i\le p-1$ depending on $B^t_\om$.
\end{itemize}
Notice that any irreducible polynomial $f\in\C[x]$ is linear, and we have a similarity of companion matrices $F_{ax+b}\sim F_{x+b/a}$, hence we can assume that any such $f$ is of the form $f=x-\xi$, with $\xi\in \C^\times$.

Let $\mathcal{A}$ be the subring of $\mathscr{A}(\omega,\mathsf{M})$ generated by $u_\xi, \xi\in\C^\times$, $x_i, y_i, y_i^\ast$, $0\le i\le p-1$, and $x_{ij}, x_{ji}$, $0\le i<j\le p-1$. We will proceed by induction on the degree of $t=(z-1)^{k_0}(z-q)^{k_1}\cdots (z-q^{p-1})^{k_{p-1}}\in \mathsf{M}$ to show that the class of any simple weight module in $A(t)\text{-wmod}_\omega$ is in $\mathcal{A}$. 
If $\deg (t)=0$, then $t=1$, which means that $B^t_\om=B^1_\om=\emptyset$. Any simple module for $t=1$, then is of the form $V^1(\om,w,f)$, with $f=x-\xi$, $\xi\in\C^\times$, and $w=1^p$, that is, its class has to be one of the $u_\xi$ so it is in $\mathcal{A}$. 

Now suppose that $\deg(t)=1$, then $t=z-q^i$ for some $0\leq i\leq p-1$. Then, since $B^t_\om=\{\Fm_i=(z-q^i)\}$ is a single break, the only possibilities for simple weight modules are $V^{t}(\om,1^{i-1}x1^{p-i},x-\xi)$, $V^{t}(\om,1^{i-1}y1^{p-i},x-\xi)$, and $V^{t}(\om,i,1^{p-1})$. By definition, $[V^{t}(\om,i,1^{p-1})]=x_i\in\mathcal{A}$, and by Lemma \ref{lem:nobr-tens-2nd} we have 
$$[V^{t}(\om,1^{i-1}x1^{p-i},x-\xi)]=u_\xi\cdot y_i\in\mathcal{A}, \quad [V^{t}(\om,1^{i-1}y1^{p-i},x-\xi)]=u_\xi\cdot y_i^*\in\mathcal{A}.$$

Next, suppose that $\deg(t)=2$. If $t=(z-q^i)^2$, then $B^t_\om=\{(z-q^i)\}$, hence, like in the $\deg(t)=1$ case, every simple $A(t)$-module has one of the following forms:
$V^{t}(\om,1^{i-1}x1^{p-i},x-\xi)$, $V^{t}(\om,1^{i-1}y1^{p-i},x-\xi)$, or $V^{t}(\om,i,1^{p-1})$.

Let $V=V^{t}(\om,i,1^{p-1})\cong V^{z-q^i}(\om,i,1^{p-1})^{\otimes 2}$ by Lemma \ref{lem:1st-tensor-1st}, then $[V]=x_i^2\in\mathcal{A}$.
If $V=V^{t}(\om,1^{i-1}x1^{p-i},x-\xi)$, then 
\begin{align*}
[V] &=u_\xi\cdot [V^{t}(\om,1^{i-1}x1^{p-i},x-1)]\\
&=u_\xi y_i^2\in\mathcal{A}.
\end{align*}
And similarly, $[V^{t}(\om,1^{i-1}x1^{p-i},x-\xi)]=u_\xi (y_i^\ast)^2\in \mathcal{A}$.

Suppose now that $t=(z-q^i)(z-q^j),\; 0\le i<j\le p-1$. By Lemma \ref{lem:2nd-kind-tensor} and Lemma \ref{lem:nobr-tens-2nd}, there exists $\xi'\in\C^\times$ such that $[V^t(\om,1^{i-1}x1^{j-i-1}x1^{p-j},x-\xi)]=u_{\xi'}y_iy_j\in\mathcal{A}$.
Similarly, $[V^t(\om,1^{i-1}y1^{j-i-1}y1^{p-j},x-\xi)]=u_{\xi'}y_i^\ast y_j^\ast\in\mathcal{A}$.
The only other simple $A(t)$-modules in this case are $x_{ij}$ and $x_{ji}$ which belong to $\mathcal{A}$ by definition.

If $\deg(t)>2$, we first consider the simple weight modules $V=V^t(\om,w,x-\xi)$ with $w_k\in\{1,x\}$ or $w_k\in\{1,y\}$ for all $k$. Write $t=(z-q^i)t'$, with $\deg(t')=\deg(t)-1$. We consider two cases, depending on whether $z-q^i$ is a factor of $t'$ or not.
\begin{itemize}
\item[Case 1)] If $z-q^i$ is a factor of $t'$, then $B^t_\om=B^{t'}_\om$. If $w_k\in\{1,x\}$ for all $k$, by Lemma \ref{lem:2nd-kind-tensor} and Lemma \ref{lem:nobr-tens-2nd},
\[[V]=u_{\xi'}\cdot [V^{t'}(\om,w,x-1)]\cdot y_i\] for some $\xi'\in\C^\times$. By the induction hypothesis, $[V^{t'}(\om,w,x-1)]\in\mathcal{A}$, hence $[V]\in\mathcal{A}$. 

Similarly, if $w_k\in\{1,y\}$ for all $k$, then \[[V]=u_{\xi'} \cdot[V^{t'}(\om,w,x-1)]\cdot y_i^\ast\] which is in $\mathcal{A}$ by induction.

\item[Case 2)] If $z-q^i$ is not a factor of $t'$, then $B^t_\om=B^{t'}_\om\sqcup \{(z-q^i)\}$. If $w_k\in\{1,x\}$ for all $k$, write $w=w_{(1)}xw_{(2)}$ where $x$ is in position $i$ and $w_{(1)}$ and $w_{(2)}$ are subwords. Let $w'=w_{(1)}1w_{(2)}$. Then $w'\in\mathbf{D}^{t'}_\om$.
By Lemmas \ref{lem:2nd-kind-tensor} and \ref{lem:nobr-tens-2nd},

\[[V]=u_{\xi'}\cdot [V^{t'}(\om,w',x-1)]\cdot y_i\]
for some $\xi'\in\mathbb{C}^\times$, which is in $\mathcal{A}$ by induction.

Similarly, if $w_k\in\{1,y\}$ for all $k$, then we write $w=w_{(1)}yw_{(2)}$, with $y$ is in position $i$, and $w'=w_{(1)}1w_{(2)}$. We then have
\[[V]=u_{\xi'}\cdot [V^{t'}(\om,w',x-1)]\cdot y_i^\ast\in\mathcal{A}.\]
\end{itemize}

Finally, again with $\deg(t)>2$, we consider the simple weight modules of the form $V=V^t(\om,i,1^k)$ for some $k\geq 0$. 

\begin{itemize}
\item[Case 1)] If $t=(z-q^i)^{\deg(t)}$ is a power of a linear factor, then $k=p-1$ necessarily and by Lemma \ref{lem:1st-tensor-1st}, $[V]=[V^{t'}(\om,i,1^{p-1})]\cdot x_i$ with $t'=(z-q^i)^{\deg(t)-1}$, which is in $\mathcal{A}$ by the induction hypothesis.    
\item[Case 2)] If $t$ has at least two distinct factors, then by definition $\si^i(\mf_0)=(z-q^i)$ and $\si^{j}(\mf_0)=(z-q^{j})$ with $j=i+k+1$ are distinct breaks in $B^t_\om$. Write $t=(z-q^i)(z-q^j)t'$, with $\deg(t')=\deg(t)-2$. Let $i', j'$ (not necessarily distinct) be such that $(z-q^{i'}), (z-q^{j'})\in B^{t'}_\om$, with $i'\leq i<j\leq j'$ in the cyclic order and such that there are no indices $\ell$ with $i'<\ell<i$ or $j<\ell<j'$ such that $(z-q^\ell)\in B^{t'}_\om$. Then, by Lemma \ref{lem:1st-tensor-1st} we have
\[[V]=\begin{cases}[V^{t'}(\om,i',1^{j'-i'-1})]\cdot x_{ij},\quad \text{ if }0\leq i'< j'\leq p-1,\\
[V^{t'}(\om,i',1^{p+j'-i'-1})]\cdot x_{ij},\quad \text{ if }0\leq j'\leq i'\leq p-1, \end{cases}\]
which in either case is in $\mathcal{A}$ by induction.
\end{itemize} 

Next we prove that relations (i)--(xii) hold.
In principle, all these relations follow directly from the lemmas in Section \ref{sec:definitions}, together with the known composition series of indecomposable weight modules classified in \cite{DGO}. For clarity, we provide some details in each case.

(i) is stating that $u_1$ is the multiplicative identity element in the Grothendieck ring. This follows from Lemmas \ref{lemma:tens-FG}, \ref{lem:2nd-kind-tensor}, and \ref{lem:1st-tensor-2nd}. 

(ii) follows from Lemma \ref{lemma:tens-FG}.

(iii) and (iv) follow from Lemma \ref{lem:1st-tensor-nob}.

(v): Use Lemma \ref{lem:1st-tensor-1st}. In more detail, let $i\neq j$ and $t=(z-q^i)(z-q^j)$ and let $V$ be any non-simple $A(t)$-module with full support and one-dimensional weight spaces. Then $V$ has length two and its simple subquotients have support in two complementary intervals with endpoints $\Fm_i$ and $\Fm_j$. Thus $[V]=x_{ij}+x_{ji}$. 
All of the isoclasses $x_i$, $y_i$, $y_i^\ast$ consist of $A(z-q^i)$-modules with full support and one-dimensional weight spaces. Therefore any product of them have the same property. Thus the $A(t)$-modules 
$y_iy_j^\ast$, $x_ix_j$, $x_iy_j$, $x_iy_j^\ast$, $y_iy_j$ and $y_i^\ast y_j^\ast$ have full support and one-dimensional weight spaces. Out of these only the last two are simple.
The other four therefore equal $x_{ij}+x_{ji}$.
This proves (v).

(vi): Use Lemma \ref{lem:1st-tensor-1st}. In more detail, let $t=(z-q^i)^2$. The isoclass $x_i^2$ consists of $A(t)$-modules $V$ with full support, one-dimensional weight spaces, and such that $X V_{\Fm_i}=0$ and $Y V_{\Fm_{i+1}}=0$. The modules in $y_iy_i^\ast$, $x_iy_i$, $x_iy_i^\ast$ have the same property due to the word rules $xy=0$, $0x=0$, $0y=0$ (see Definition \ref{def:monoid}), respectively. 

(vii): Let $t=(z-q^i)(z-q^j)(z-q^k)$, $i<k<j$. By Lemma \ref{lem:1st-tensor-1st}, $x_{ij}x_k$ decomposes as a direct sum of two simple modules, one with support $\{\Fm_a\mid i<a\le k\}$ and the other with support $\{\Fm_a\mid k<a\le j\}$. There is only one simple $A(t)$-module for each of those supports, corresponding to $x_{ik}x_j$ and $x_{kj}x_i$ respectively.

(viii): The supports of $x_{ij}$ and $x_{kl}$ are disjoint, therefore their product is zero, by Lemma \ref{lemma:weightmod}.

(ix): The support of $x_{ij}x_{kl}$ is the intersection of the supports of $x_{ij}$ and $x_{kl}$. The support of $x_{ij}$ is $[i,j):=\{\Fm_i,\Fm_{i+1},\ldots,\Fm_{j-1}\}$. So the support of $x_{ij}x_{kl}$ equals $[i,j)\cap [k,l)=[k,j)$ when $i\le k<j\le\ell\le i$. Meanwhile, $x_i$ and $x_\ell$ have full support. Multiplying by them does not change the support, only the degree. This shows that $x_{ij}x_{kl}$ and $x_{kj}x_ix_\ell$ have the same support, one-dimensional weight spaces, and are of the same degree. Therefore they are equal.

(x): The proof is similar to the proof of relation (ix).

(xi): The proof is similar to the proof of relation (ix). Observe that in this case the support of either side equals $[i,\ell)\sqcup [k,j)$ as the two intervals overlap twice in the left hand side.

(xii): The proof is again similar to the proof of relation (ix): $y_k, y_k^\ast, x_k$ all have full support. However, when multiplying by $x_{ij}$ the support becomes $[i,j)$ in all three cases.

It remains to show that this is a complete set of relations.

We now show that, using the relations (i)-(xii), we can write any monomial in the generators $u_\xi$, $x_i$, $x_{ij}$, $y_i$, $y^*_i$ as a linear combination of the following:

\begin{enumerate}
\item\label{it:1} $x_{ij}x_{k_1}^{a_1}\cdots x_{k_r}^{a_r}$, with $a_\ell\geq0$ and $i\neq j$ and $j\leq k_\ell\leq i$ in the cyclic order for all $\ell$;
\item\label{it:2} $x_i^{a}$, $a>0$;
\item\label{it:3} $u_\xi$, $\xi\in\C^\times$;
\item\label{it:4} $u_\xi y_{k_1}^{a_1}\cdots y_{k_r}^{a_r}$, with $\xi\in\C^\times$, $a_\ell>0$ for all $\ell$;
\item\label{it:5} $u_\xi (y^*_{k_1})^{a_1}\cdots (y^*_{k_r})^{a_r}$, with $\xi\in\C^\times$, $a_\ell>0$ for all $\ell$.
\end{enumerate}

First of all, suppose that the monomial contains a generator $x_{ij}$ for some $i$ and $j$. Then using (iv) we can make it so that the monomial does not contain any generator $u_\xi$ for any $\xi$, and using (xii) we can also write the monomial as not containing any $y_k$ nor $y^*_k$ but only potentially terms of the form $x_k$. Also, we can use relations (viii)-(xi) to have only a single generator of the form $x_{ij}$ in the monomial. Finally, we can use (vii) to make sure that any $x_k$'s appearing are such that $i<j\leq k\leq i$ in the cyclic order. Hence, any monomial containing $x_{ij}$ can be written as sum of terms of the form \eqref{it:1}.

Now, suppose that we have a monomial that contains $x_i$ for some $i$. By relation (v), if it also contains any of $x_j$, $y_j$, or $y^*_j$ with $j\neq i$, then we can write it as a sum of monomials each including an $x_{ij}$, hence we are in the previous case. It follows that any monomial containing $x_i$ that cannot be written as in \eqref{it:1} has to contain only terms $x_i$, $y_i$ and $y^*_i$ with the same index. Then, by (iii) we can remove any term $u_\xi$, for any $\xi\in\C^\times$ and by (vi) we can rewrite such a monomial as just a power of $x_i$, so it is of the form \eqref{it:2}.

All other monomials not containing generators $x_{ij}$ nor $x_k$ for any indices have to just be expressions in the $u_\xi$'s, $y_i$'s and $y^*_i$'s. Notice that a monomial containing both a generator $y_i$ and $y^*_j$ can be written with $x_{ij}$ (if $i\neq j$ using (v)) or with a power of $x_i$ (if $i=j$ using (vi)). Also if we have multiple terms of the form $u_\xi$, we can consolidate them into a single one by using (ii). It follows that the only other possibilities are either a single $u_\xi$ \eqref{it:3} , or $u_\xi$ times a product of $y_i$'s \eqref{it:4}, or $u_\xi$ times a product of $y^*_i$'s  \eqref{it:5}.

Now, considering the monomials in normal form (1)-(5), we have, using Definition \ref{def:generators}, together with Lemmas \ref{lem:1st-tensor-1st}, \ref{lem:2nd-kind-tensor}, and \ref{lem:nobr-tens-2nd}, the following equalities in the Grothendieck ring:

\begin{enumerate}
\item $x_{ij}x_{k_1}^{a_1}\cdots x_{k_r}^{a_r}=[V^t(\om,i,1^{j-i})]$, with $t=(z-q^i)(z-q^j)\prod_{\ell=1}^r(z-q^{k_\ell})^{a_\ell}$;
\item $x_i^a=[V^t(\om,i,1^{p-1})]$, with $t=(z-q^i)^a$;
\item $u_{\xi}=[V^1(\om,1^p,x-\xi)]$;
\item $u_\xi y_{k_1}^{a_1}\cdots y_{k_r}^{a_r}=[V^t(\om,w,x-\xi\nu_{t})]$, where $t=\prod_{\ell=1}^r(z-q^{k_\ell})^{a_\ell}$, $w=w_1\cdots w_p$ with $w_j\in\{1,x\}$ (depending on the zeroes of $t$) for all $j=1,\ldots,p$, and $\nu_{t}\in\C^\times$;
\item $ u_\xi (y^*_{k_1})^{a_1}\cdots (y^*_{k_r})^{a_r}=[V^t(\om,w,x-\xi)]$, where $t=\prod_{\ell=1}^r(z-q^{k_\ell})^{a_\ell}$, and $w=w_1\cdots w_p$ $w_j\in\{1,y\}$ (depending on the zeroes of $t$) for all $j=1,\ldots,p$.
\end{enumerate}

These are distinct classes of simple modules, hence they are linearly independent in the sum of Grothendieck rings. This proves that it is a basis and that we have found a complete set of relations.
\end{proof}

\begin{Remark}
It follows from the linear independence of the words in normal form in the above proof that the subalgebra $\C[u_\xi\mid \xi\in\C^\times]$ of $\mathscr{A}(\omega,\mathsf{M})$ is isomorphic to the group algebra $\C[\C^\times]$ of the multiplicative group of nonzero complex numbers.
\end{Remark}

\begin{Proposition}
For a positive integer $p$, let $\mathscr{A}(p)$ be the Grothendieck ring $\mathscr{A}(\omega,\mathsf{M})$ modulo the ideal generated by $u_\xi-1$ for all $\xi\in\C^\times$. $\mathscr{A}(p)$ is graded by the commutative monoid $\mathsf{M}$ which we identify with $\N^p$.
\begin{enumerate}[{\rm (a)}]
\item We have
\begin{equation}\label{eq:dimAd}
\sum_{d\in\mathbb{N}^p} (\dim \mathscr{A}(p)_d)T^d = -1+\frac{2+T_1+T_2+\cdots+T_p}{(1-T_1)(1-T_2)\cdots(1-T_p)}.
\end{equation}
\item The Hilbert series of $\mathscr{A}(p)$ equals
\begin{equation}
H_{\mathscr{A}(p)}(T)=-1+\frac{2+pT}{(1-T)^p}.
\end{equation}
\end{enumerate}
\end{Proposition}

\begin{proof}
Fix $p$ and put $A=\mathscr{A}(p)$ for brevity.

(a) Obviously $\dim A_0=1$.
By Theorem \ref{thm:groth-ring}, if $d$ is nonzero, a basis for $A_d$ is given by 
\[\big\{y_1^{d_1}y_2^{d_2}\cdots y_p^{d_p},\, 
(y_1^\ast)^{d_1}(y_2^\ast)^{d_2}\cdots (y_p^\ast)^{d_p}\big\}\cup\{ z_{d,i}\mid 1\le i\le p, d_i\neq 0\big\}\]
where for each $i$ such that $d_i\neq 0$, $z_{d,i}$ is defined as follows:
\begin{itemize}
\item if $i$ is the unique index such that $d_i\neq 0$, we have $z_{d,i}=x_i^{d_i}$;
\item otherwise, $z_{d,i}=x_{ij}x_{k_1}^{a_1}\cdots x_{k_r}^{a_r}$, with $a_\ell\geq0$ and $i\neq j$ and $j\leq k_\ell\leq i$ in the cyclic order for all $\ell$, where the $a_s$ are determined by the requirement  $\deg(z_{d,i})=d$, and $j$ is the unique index such that $d_r=0$ for all $r$ with $i<r<j$.
\end{itemize}
Thus
\[
\dim A_d = \begin{cases} 2+\delta_{d_1>0}+\cdots+\delta_{d_p>0},& d\neq 0,\\
1,& d=0,\end{cases}
\]
where $\delta_P$ equals $1$ if statement $P$ is true and $0$ otherwise.
Thus, using multi-index notation $T^d=T_1^{d_1}\cdots T_p^{d_p}$,
\begin{align*}
1+\sum_{d\in\N^p}(\dim A_d)T^d 
&=\sum_{d\in\N^p}(2+\delta_{d_1>0}+\cdots+\delta_{d_p>0})T^d\\
&=2\sum_{d\in\N^p}T^d + \sum_{d\in\N^p}\delta_{d_1>0}T^d+\cdots
+\sum_{d\in\N^p}\delta_{d_p>0}T^d\\
&=2\sum_{d\in\N^p}T^d + \sum_{d\in\N^p}T_1T^d+\cdots+\sum_{d\in\N^p}T_pT^d\\
&=(2+T_1+T_2+\cdots+T_p)
\big(\sum_{d_1\in\N}T_1^{d_1}\big)\cdots \big(\sum_{d_p\in\N}T_p^{d_p}\big)\\
&=\frac{2+T_1+T_2+\cdots+T_p}{(1-T_1)(1-T_2)\cdots(1-T_p)}
\end{align*}

(b) Immediate by substituting $T_1=T_2=\cdots=T_p=T$ in \eqref{eq:dimAd}.
\end{proof}

\begin{Remark}
Experimental calculations indicate that that for every $p\ge 1$, the power series expansion of $P_{\mathscr{A}(p)}(T):=1/H_{\mathscr{A}(p)}(-T)=\sum_{d=0}^\infty e_d(p) T^d$ at $T=0$ has non-negative integer coefficients $e_d(p)$. 
(If $\mathscr{A}(p)$ were a quadratic algebra this would provide evidence for $\mathscr{A}(p)$ being a Koszul algebra.)
We currently do not understand why. It would be interesting to interpret $e_d(p)$ as the dimension of some vector space related to $\Ext^d(\mathscr{A}(p),\mathscr{A}(p))$.
\end{Remark}

\section{On the structure of the split Grothendieck ring}\label{sec:split-groth-ring}

Let $R$ be commutative ring, $\si\in\Aut(R)$ of order $p\in\Z_{>0}$, $\omega\in \MaxSpec(R)/\Z$ also of order $p$.
Let $\mathsf{M}\subset R_{\mathrm{reg}}$ be a submonoid of the regular elements of $R$. For $t\in\mathsf{M}$, let $A(t)=A(R,\si,t)$ be the GWA and 
\[\mathscr{A}^{\mathrm{split}}(\omega,\mathsf{M})=\bigoplus_{t\in\mathsf{M}} \C\otimes_\Z K_0^{\mathrm{split}}\big(A(t)\mathrm{-wmod}_\omega\big)\]
be the corresponding split Grothendieck ring. In general, this is hard to describe, but we are able to obtain some partial results about its structure in some cases. Recall that a basis of the split Grothendieck group is given by the classes of non-isomorphic indecomposable modules.

\subsection{General properties}

\begin{Definition}\label{def:ideal}
Let $\mathscr{I}(\om,\mathsf{M})$ be the span of all equivalence classes $[V]\in\mathscr{A}^{\mathrm{split}}(\omega,\mathsf{M})$ of indecomposable weight modules of the kind $V=V^t(\om,i,w)$ for $t\in \mathsf{M}$ such that $B^t_\omega\neq\emptyset$ and $i,w$ are as in Definition \ref{def:type1mod}.
\end{Definition}

Recall that $\Bbbk$ denotes the field $R/\Fm_0$, where $\Fm_0$ is a fixed but arbitrary element of the orbit $\omega$.

\begin{Proposition} \label{prop:ideal}
Assume $\Bbbk$ is algebraically closed. Then 
 $\mathscr{I}(\om,\mathsf{M})$ is an ideal of $\mathscr{A}^{\mathrm{split}}(\om,\mathsf{M})$.
\end{Proposition}

\begin{proof}
By Lemma \ref{lem:1st-tensor-1st}, $\mathscr{I}(\om,\mathsf{M})$ is closed under multiplication.
The class of any indecomposable module which is not in $\mathscr{I}(\om,\mathsf{M})$ is of the form $[V^t(\om,w,f)]$, where $f$ is an indecomposable polynomial in $\Bbbk[x]$.
Since $\Bbbk$ is algebraically closed, 
there is an $f'\in\Bbbk[x]$ such that $f'^{[r]}=f$, with the property that
$V^t(\om,w,f)\simeq V^1(\om,1^p,f')\otimes_R V^t(\om,w,x-1)$ by Corollary \ref{cor:tens-f-x-1}.
By Lemmas \ref{lem:1st-tensor-2nd}, and \ref{lem:1st-tensor-nob} it now follows that $[V^t(\om,w,f)]\cdot \mathscr{I}(\om,\mathsf{M})\subseteq \mathscr{I}(\om,\mathsf{M})$.
\end{proof}

\begin{Remark}
We expect this result to hold even when the ground field is not algebraically closed.
\end{Remark}

\subsection{Split Grothendieck ring for trivial monoid}\label{sec:triv-monoid}

We will use the same set up for $R$, $\si$ and $\om$ as we did in Section \ref{sec:Groth-ring}, which we now recall. We let $R=\C[z,z^{-1}]$, $\si(z)=q^{-1}z$ with $q=e^{\frac{2\pi i}{p}}$ a primitive $p$-th root of unity, $p\in\Z_{>0}$. We choose $\mf_0 = (z-1)$, so that we have the orbit $\om=\{(z-q^k)\}_{k=0}^{p-1}$, with $\Fm_i=\si^i(\mf_0)=(z-q^i)$. However, here we consider the trivial monoid $\mathsf{M}=\{1\}$, so in the sum defining the algebra we have just one Grothendieck group 

\[\mathscr{A}^{\mathrm{split}}(\omega,\{1\})=\C\otimes_\Z K_0^{\mathrm{split}}\big(A(1)\mathrm{-wmod}_\omega\big).\]

When $t=1$, the generalized Weyl algebra is isomorphic to a skew Laurent polynomial algebra over $R$, hence $A(1)\simeq \C[z,z^{-1}][X,X^{-1};\si]$, $\si(z)=q^{-1}z$, where $q=e^{\frac{2\pi i}{p}}$.

By the classification in \cite{DGO}, as explained in Remark \ref{rem:indec-nobr}, all indecomposable weight modules for $A(1)$ with support in $\omega$ are of the form $V^1(\om,1^p,f)$, where $f$ is an indecomposable polynomial in $\Bbbk[x]=\C[x,x^{-1}]$. Such an $f$ has then to be of the form $f=(x-\xi)^a$, with $\xi\in\C^\times$ and $a\geq 1$.

\begin{Definition}We define the elements
\[u(\xi,a)=\big[V^1(\om,1^p,(x-\xi)^a)\big],\qquad \xi\in\C^\times, a\in\Z_{\ge 1},\]
which form a basis of $\mathscr{A}^{\mathrm{split}}(\omega,\{1\})$.
\end{Definition}

\begin{Lemma}\label{lem:tensor-jordan}Multiplication in the split Grothendieck ring $\mathscr{A}^{\mathrm{split}}(\omega,\{1\})$, with respect to the basis we just defined, is given by
\[u(\xi,a)u(\eta,b)=\sum_{k=1}^{\min\{a,b\}} u\big(\xi\eta,\,a+b-(2k-1)\big)\]
for all $\xi,\eta\in\C^\times$, $a,b\in\Z_{\ge 1}$.
\end{Lemma}

\begin{proof}
Since the linear transformation $F_{(x-\xi)^a}$ given by the companion matrix of the indecomposable polynomial $(x-\xi)^a$ is isomorphic to a single Jordan block of size $a$ with eigenvalue $\xi$, this follows immediately from the Jordan decomposition of the tensor product of two Jordan blocks, which can be found in \cite[Thm 4.3.17]{HJ}.
\end{proof}

The following result follows immediately from \ref{lem:tensor-jordan} and the properties of the representations of the Lie algebra $\Fsl_2$.

\begin{Proposition}
The split Grothendieck ring $\mathscr{A}^{\mathrm{split}}(\omega,\{1\})$ contains a copy of the  Grothendieck ring of the (semisimple) category of finite-dimensional representations of the Lie algebra $\Fsl_2$. For $n\in\Z_{\ge 0}$, let $L(n)$  denote the $(n+1)$-dimensional irreducible representation of $\Fsl_2$, then an embedding is given by
\[[L(n)]\mapsto u(1,n+1),\qquad\forall n\in\Z_{\ge 0}.\]
\end{Proposition}

\begin{Remark}
Note that the category of weight modules over the skew Laurent polynomial ring is not semisimple. So the category $A(1)\mathrm{-wmod}_\omega$ is not equivalent to the category of finite-dimensional representations of $\mathfrak{sl}_2$.
\end{Remark}

\begin{Theorem}\label{thm:nobr-split-ring}
$\mathscr{A}^{\mathrm{split}}(\omega,\{1\})$ is generated by \begin{equation}\label{eq:trivial-monoid-generators}
\{u(\xi,1)\mid \xi\in\C^\times\}\cup \{u(1,2)\}
\end{equation}
subject only to the relations
\begin{equation}\label{eq:prod-c-times}
u(\xi,1)u(\eta,1)=u(\xi\eta,1)\quad\text{ for all }\xi,\eta\in\C^\times.
\end{equation}
In other words, $\mathscr{A}^{\mathrm{split}}(\omega,\{1\})$ is isomorphic to the algebra of polynomials in one indeterminate over the group algebra of $\C^\times$:
\begin{equation}
\mathscr{A}^{\mathrm{split}}(\omega,\{1\})\cong \C[\C^\times][u(1,2)]
\end{equation}
\end{Theorem}

\begin{proof}
First of all, by induction on $a\ge 1$ we show that $u(1,a)$ belongs to the subalgebra generated by \eqref{eq:trivial-monoid-generators}. For $a=1,2$, $u(1,a)$ is one of the generators, so this is clear. If $a>2$, then by Lemma \ref{lem:tensor-jordan} we have that 
$$u(1,a-1)u(1,2)=u(1,a)+u(1,a-2),$$ so 
$$u(1,a)=u(1,a-1)u(1,2)-u(1,a-2)$$ 
is in the subalgebra by the inductive hypothesis. Then, again by Lemma \ref{lem:tensor-jordan}, we have that $u(\xi,a)=u(\xi,1)u(1,a)$ which shows that the elements in \eqref{eq:trivial-monoid-generators} do generate the whole $\mathscr{A}^{\mathrm{split}}(\omega,\{1\})$. 

The fact that relation \eqref{eq:prod-c-times} holds follows immediately from Lemma \ref{lem:tensor-jordan} and, given that, for different $\xi\in \C^\times$, $u(\xi,1)$ are classes of non-isomorphic simple (hence indecomposable) modules, we get that the subalgebra generated by $u(\xi,1)$, $\xi\in\C^\times$ is indeed isomorphic to the group algebra $\C[\C^\times]$.

Finally, we only need to observe that different powers of $u(1,2)$ are linearly independent in $\mathscr{A}^{\mathrm{split}}(\omega,\{1\})$ and this is true because by using Lemma \ref{lem:tensor-jordan} several times, we can write 
$$u(1,2)^n=u(1,n+1)+\text{lower terms},$$ 
where the lower terms only involve $u(1,a)$ with $a\leq n$.
\end{proof}

\subsection{Split Grothendieck ring for single break monoids}

As in Section \ref{sec:triv-monoid}, we let $R=\C[z,z^{-1}]$, $\si(z)=q^{-1}z$ with $q=e^{\frac{2\pi i}{p}}$ a primitive $p$-th root of unity, $p\in\Z_{>0}$. We choose $\mf_0 = (z-1)$, so that we have the orbit $\om=\{(z-q^k)\}_{k=0}^{p-1}$, with $\Fm_i=\si^i(\mf_0)=(z-q^i)$.

In this section we restrict our attention to the case of at most one break. More specifically, fix $i\in\{0,1,\ldots,p-1\}$ and let $\mathsf{M}_i$ be the multiplicative submonoid of $R$ generated by $\{z-q^i\}$. Thus if $t\in\mathsf{M}_i$ then $t=(z-q^i)^k$ for some non-negative integer $k$, hence we consider the split Grothendieck ring

\[\mathscr{A}^{\mathrm{split}}(\omega,\mathsf{M}_i)=\bigoplus_{t\in\mathsf{M}_i}\C\otimes_\Z K_0^{\mathrm{split}}\big(A(t)\mathrm{-wmod}_\omega\big)=\bigoplus_{k\geq 0}\C\otimes_\Z K_0^{\mathrm{split}}\big(A((z-q^i)^k)\mathrm{-wmod}_\omega\big).\]

Even describing this algebra is quite involved and we will only obtain partial results.

It is easy to see that for any $i=1,\ldots,p-1$, $\mathscr{A}^{\mathrm{split}}(\omega,\mathsf{M}_i)$ is isomorphic to $\mathscr{A}^{\mathrm{split}}(\omega,\mathsf{M}_0)$, and therefore we assume $i=0$ in the rest of this section.

For $j\ge 1$, let
\begin{equation}
\mathscr{A}_{(j)}=\mathscr{A}^{\mathrm{split}}(\omega,\mathsf{M}_0)_{(j)}
\end{equation}
be the subspace of the split Grothendieck ring spanned by equivalence classes of indecomposable weight modules of length at most $j$. In particular, $\mathscr{A}_{(1)}$ is spanned by the classes of simple weight modules.

\begin{Theorem}\label{thm:section}
The subspace $\mathscr{A}_{(1)}$ is a subalgebra of the split Grothendieck ring $\mathscr{A}^{\mathrm{split}}(\omega,\mathsf{M}_0)$. As a consequence, the canonical surjective algebra map $\pi:\mathscr{A}^{\mathrm{split}}(\omega,\mathsf{M}_0)\ronto\mathscr{A}(\omega,\mathsf{M}_0)$ has a section which gives an isomorphism of algebras $ \mathscr{A}(\omega,\mathsf{M}_0)\simeq \mathscr{A}_{(1)}$.
\end{Theorem}

\begin{proof} We can define a map $\alpha:\mathscr{A}(\omega,\mathsf{M}_0)\to \mathscr{A}^{\mathrm{split}}(\omega,\mathsf{M}_0)$ by sending a module to the class of the direct sum of its simple subquotients. It is clear that $\alpha$ is a well defined linear map on the Grothendieck group, $\Im(\alpha)=\mathscr{A}_{(1)}$, and $\pi\circ \alpha=\Id$, so it gives a section of the canonical surjective map. Therefore, the result follows once it is established that $\mathscr{A}_{(1)}$ is a subalgebra, for which it suffices to show that the tensor product of two simple weight modules is semisimple. Since the only break is $(z-1)$, from the classification in \cite{DGO}, all the simple weight modules are of one of the following forms (we will use variables to denote actual modules instead of classes):
\begin{align*}
& u_\xi=V^1(\om,1^p,x-\xi), & \xi\in\C^\times,\\
& x_{a}=V^{(z-1)^a}(\om,0,1^{p-1}), & a\geq 1,\\
& y_{a,\xi}=V^{(z-1)^a}(\om,1^{p-1}x,x-\xi), & a\geq 1,~\xi\in\C^\times,\\
& y^*_{a,\xi'}=V^{(z-1)^a}(\om,1^{p-1}y,x-\xi), & a\geq 1,~\xi\in\C^\times.
\end{align*}
Then, using Lemma \ref{lemma:tens-FG}, Corollary \ref{cor:tens-f-x-1}, and Lemma \ref{lem:2nd-kind-tensor}, we have $u_\xi\otimes_R u_{\xi'}\simeq u_{\xi\xi'}$, $u_\xi\otimes_R y_{a,\xi'}\simeq y_{a,\xi\xi'}$, $u_\xi\otimes_R y^*_{a,\xi'}\simeq y^*_{a,\xi\xi'}$, $y_{a,\xi}\otimes_R y_{b,\xi'}\simeq y_{a+b,\xi\xi'}$, $y^*_{a,\xi}\otimes_R y^*_{b,\xi'}\simeq y^*_{a+b,\xi\xi'}$, $y_{a,\xi}\otimes_R y^*_{b,\xi'}\simeq x_{a+b}$. From Lemma \ref{lem:1st-tensor-1st}  we get that $x_a\otimes_R x_b\simeq x_{a+b}$, and from Lemma \ref{lem:1st-tensor-nob} we have $u_\xi\otimes_R x_a\simeq x_a$. Finally, using Corollary \ref{cor:tens-f-x-1}, together with Lemma \ref{lem:1st-tensor-2nd}, we get that $x_a\otimes_R y_{b,\xi}\simeq x_{a+b}\simeq x_a\otimes_R y^*_{b,\xi}$. Hence each pairwise tensor product of two simple weight modules is again simple.
\end{proof}

\begin{Remark}
The subspaces $\mathscr{A}_{(j)}$ where $j\ge 2$ are not subalgebras of the split Grothendieck ring, for example $V^1(\om,1^p,(x-1)^2)$ has length $2$, but $$V^1(\om,1^p,(x-1)^2)\otimes_R V^1(\om,1^p,(x-1)^2)\simeq V^1(\om,1^p,(x-1)^3)\oplus V^1(\om,1^p,x-1)$$
so we can get indecomposable elements of length $3$ as tensor products of modules of length $2$.
\end{Remark}

\begin{Remark}\label{rem:tensor-products-more-breaks}
In the case of the choice of a monoid that has elements with more than a single break on the orbit, it is not true that tensor products of simple weight modules remains semisimple.
For example, let $p=3$, $t=(z-1)$, $t'=(z-q)(z-q^2)$, then we have two simple modules $V^t(\om, w,x-1)$, where $w=11x$, and $V^{t'}(\om,2,1)$, but then their tensor product is $V^{tt'}(\om,2,x)$ which is indecomposable but not semisimple (it is a nontrivial extension of $V^{tt'}(\om,2,\emptyset)$ and $V^{tt'}(\om,0,\emptyset)$). See Figure \ref{fig:simple-tensor-simple-is-not-semisimple}.

\begin{figure}
\[
\begin{tikzpicture}[baseline=-.5ex, vertex/.style={draw,circle,fill=white,minimum size=5pt, inner sep=0pt}]
\node[vertex] (A) at (1,0) {};
\node[vertex] (F) at ({cos(240)},{sin(240)}) {}; \node[vertex] (H) at ({cos(120)},{sin(120)}) {}; 
\draw[<->] (F) edge[out=330, in=270] (A);
\draw[->] (A) edge[out=90, in=30] (H);
\draw[<->] (H) edge[out=210, in=150] (F);
\end{tikzpicture}
%
\;\;\otimes\;\;
\begin{tikzpicture}[baseline=-.5ex, vertex/.style={draw,circle,fill=white,minimum size=5pt, inner sep=0pt}]
\node[vertex] (A) at (1,0) {};
\node[vertex] (H) at ({cos(120)},{sin(120)}) {}; 
\draw[<->] (A) edge[out=90, in=30] (H);
\end{tikzpicture}
%
%
\;\;=\;\;
\begin{tikzpicture}[baseline=-.5ex, vertex/.style={draw,circle,fill=white,minimum size=5pt, inner sep=0pt}]
\node[vertex] (A) at (1,0) {};
\node[vertex] (H) at ({cos(120)},{sin(120)}) {}; 
\draw[->] (A) edge[out=90, in=30] (H);
\end{tikzpicture}
%
\]
\caption{When we allow orbits with more than one break, the tensor products of two simple modules need not be semisimple.}
\label{fig:simple-tensor-simple-is-not-semisimple}
\end{figure}
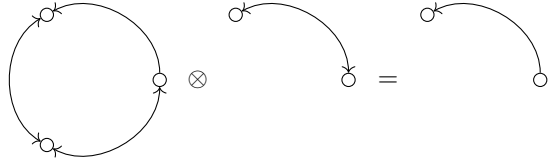
\end{Remark}

Recall the ideal $\mathscr{I}(\om,\mathsf{M}_0)$ from Definition \ref{def:ideal}, the following theorem gives a description of the structure of the quotient ring $\mathscr{A}^{\mathrm{split}}(\omega,\mathsf{M}_0)/\mathscr{I}(\om,\mathsf{M}_0)$. Recall also the action of $\Z$ by cyclic shifts on the subset of $\boldsymbol{D}_\om^{(z-1)}$ consisting of $p$-words (words whose length is a multiple of $p$) from Definition \ref{def:word-shift}. We let $\boldsymbol{D}_{\om,\operatorname{np}}^{(z-1)}\subset\boldsymbol{D}_\om^{(z-1)}$ be the subset of $p$-words that do not contain any letter $0$ and are non-periodic. We observe that the $\Z$-action by cyclic shift restricts to an action on $\boldsymbol{D}_{\om,\operatorname{np}}^{(z-1)}$. If $w\in\boldsymbol{D}_{\om,\operatorname{np}}^{(z-1)}$, we will denote by $\overline{w}$ its $\Z$-orbit in $\boldsymbol{D}_{\om,\operatorname{np}}^{(z-1)}/\Z$.

\begin{Theorem}\label{thm:split-quotient}
The quotient algebra $\mathscr{A}^{\mathrm{split}}(\omega,\mathsf{M}_0)/\mathscr{I}(\om,\mathsf{M}_0)$ is isomorphic to the commutative $\C$-algebra with generators
$u_\xi$, $\xi\in\C^\times$, $u(1,2)$, and $y_{\overline{w}}$, $\overline{w}\in \boldsymbol{D}_{\om,\operatorname{np}}^{(z-1)}/\Z$ subject to the relations
\begin{enumerate}[{\rm (i)}]
\item $u_1=1$,
\item $u_\xi u_{\xi'}=u_{\xi\xi'}$ for all $\xi,\xi'\in\C^\times$,
\item $u_\xi y_{\overline{w}}=y_{\overline{w}}$ if $w=w_1\cdots w_{rp}\in\boldsymbol{D}_{\om,\operatorname{np}}^{(z-1)}/\Z$ and $\xi^r=1$,
\item $y_{\overline{w}} y_{\overline{w}'}=0$ if $\overline{w}\neq \overline{w}'$.
\end{enumerate}
\end{Theorem}

\begin{proof}
By the classification in \cite{DGO}, all the classes of indecomposable weight modules in $\mathscr{A}^{\mathrm{split}}(\omega,\mathsf{M}_0)$ that are not in $\mathscr{I}(\om,\mathsf{M}_0)$ are of the form $[V^{(z-1)^n}(\om,w,(x-\xi)^a)]$, with $n\geq 0$, $a\geq 1$, and $w\in\boldsymbol{D}_{\om, \operatorname{np}}^{(z-1)^n}$ a non-periodic $p$-word that contains no letters $w_j=0$.
Since $V^{(z-1)^n}(\om,w,(x-\xi)^a)\simeq V^{(z-1)^n}(\om,w',(x-\xi)^a)$ if and only if $w'=w[j]$ for some $j\in \Z$, it is clear that, if we pick an arbitrary representative $w$ for each orbit $\overline{w}\in \boldsymbol{D}_{\om,\operatorname{np}}^{(z-1)^n}/\Z$, the set 
$$\{[V^{(z-1)^n}(\om,w,(x-\xi)^a)]~|~n\geq 0, ~a\geq 1,~\xi\in\C^\times,~\overline{w}\in \boldsymbol{D}_{\om,\operatorname{np}}^{(z-1)^n}/\Z \} $$
is a basis of the quotient $\mathscr{A}^{\mathrm{split}}(\omega,\mathsf{M}_0)/\mathscr{I}(\om,\mathsf{M}_0)$.

We claim that the quotient is generated as an algebra by the following elements:
\begin{align*}
u_\xi&=[V^{1}(\om,1^p,x-\xi)], ~\xi\in\C^\times, \\
y_{w}&=[V^{(z-1)}(\om,w,x-1)], ~\bar{w}\in \boldsymbol{D}_{\om, \operatorname{np}}^{(z-1)^n}/\Z,\\
u(1,2)&=[V^{1}(\om,1^p,(x-1)^2)].
\end{align*}

First of all, if $n=0$, then the only possible $w\in\boldsymbol{D}_{\om,\operatorname{np}}^{(z-1)^n}$ is $w=1^p$, so we have
$$V^{(z-1)^0}(\om,w,(x-\xi)^a)=V^{1}(\om,1^p,(x-\xi)^a)\simeq V^{1}(\om,1^p,x-\xi)\otimes_R V^{1}(\om,1^p,(x-1)^a) $$
by Lemma \ref{lemma:tens-FG}, hence these modules are generated in the Grothendieck ring by $u_\xi$, $\xi\in\C^\times$ and $u(1,2)$ exactly as in the Proof of Theorem \ref{thm:nobr-split-ring}.

If $n\geq 1$, by Corollary \ref{cor:tens-f-x-1}, if the length of $w$ is $rp$, $r\in\Z_{\geq 1}$, then we have
$$V^{(z-1)^n}(\om,w,(x-\xi)^a)\simeq V^{1}(\om,1^p,(x-\xi^{1/r})^a)\otimes_R V^{(z-1)^n}(\om,w,x-1) $$
where $\xi^{1/r}$ is any of the $r$-th roots of $\xi$. Then, by Lemma \ref{lemma:tens-FG} we have that
$$V^{1}(\om,1^p,(x-\xi^{1/r})^a)\simeq V^{1}(\om,1^p,x-\xi^{1/r})\otimes_R V^{1}(\om,1^p,(x-1)^a). $$
Hence
\begin{equation}\label{eq:decomp-uxi-a-n}[V^{(z-1)^n}(\om,w,(x-\xi)^a)]=u_{\xi^{1/r}}\cdot [V^{1}(\om,1^p,(x-1)^a)]\cdot [V^{(z-1)^n}(\om,w,x-1)].\end{equation}

Now we can use Lemma \ref{lem:2nd-kind-tensor} to compute 
\begin{equation*}V^{(z-1)}(\om,w,x-1)\otimes_R V^{(z-1)}(\om,w,x-1)\simeq \bigoplus_{j=0}^{r-1}V^{(z-1)^2}(\om,w\otimes (w[j]),x-(z-1)^{w,w[j]}(z-1)^{w[j],w})\end{equation*}

Notice that if $w\in\boldsymbol{D}_{\om,\operatorname{np}}^{(z-1)}$, with the length of $w$ being $rp$, then $w_i\neq 1$ if and only if $i=kp$ with $k=1,\ldots, r$. Then, for any shift we have that either $w[j]=w$, or the word $w\otimes (w[j])$ will contain a letter $0$ because if $w[j]\neq w$ that means that there is an index $i$ such that $w_i=x$ and $(w[j])_i=y$ or viceversa. Further, since $w$ is nonperiodic, the only index $j=0,\ldots, r-1$ such that $w=w[j]$ is $j=0$. Hence, by Lemma \ref{lemma:fakecirc}, for all $j=1,\ldots, r-1$, we have $[V^{(z-1)^2}(\om,w\otimes (w[j]),x-(z-1)^{w,w[j]}(z-1)^{w[j],w})]\in\mathscr{I}(\om,\mathsf{M}_0).$ Finally, we have that when $j=0$, $w\otimes (w[j])=w\otimes w=w$ and
$(z-1)^{w,w[j]}=(z-1)^{w,w}=1=(z-1)^{w[j],w}$,
so in the quotient $\mathscr{A}^{\mathrm{split}}(\omega,\mathsf{M}_0)/\mathscr{I}(\om,\mathsf{M}_0)$ we have
\begin{equation}\label{eq:z-1-squared}[V^{(z-1)}(\om,w,x-1)]^2=[V^{(z-1)}(\om,w,x-1)\otimes _R V^{(z-1)}(\om,w,x-1)]=[V^{(z-1)^2}(\om,w,x-1)].\end{equation}
By induction on $n$, using \eqref{eq:z-1-squared}, we then get that for all $n\geq 1$ and for all $w\in\boldsymbol{D}_{\om,\operatorname{np}}^{(z-1)}$
\begin{equation}\label{eq:z-1-tothe-n}[V^{(z-1)^n}(\om,w,x-1)]=[V^{(z-1)}(\om,w,x-1)]^n=(y_{\overline{w}})^n.\end{equation}
Putting together \eqref{eq:decomp-uxi-a-n}, \eqref{eq:z-1-tothe-n} and the same argument as in the Proof of Theorem \ref{thm:nobr-split-ring} which shows that any $[V^{1}(\om,1^p,(x-1)^a)]$ is a polynomial in $u(1,2)$, the statement about generators is proved.

Relations (1) and (2) follow in the same way as in the proof of Theorem \ref{thm:groth-ring}: $u_1$ is the multiplicative identity element in the Grothendieck ring from Lemmas \ref{lemma:tens-FG} (which also implies (2) ), \ref{lem:2nd-kind-tensor}, and \ref{lem:1st-tensor-2nd}. Relation (3) follows from Corollary \ref{cor:tens-f-x-1}. Relation (4) follows from Lemma \ref{lem:2nd-kind-tensor} and Lemma \ref{lemma:fakecirc}, since the word $w\otimes (w'[j])$ will always contain a letter of $0$, unless $w=w'[j]$ for some $j$, which means exactly that $\overline{w}=\overline{w}'$. So all the relations are satisfied.

Finally, we show that we found all the relations. Using relation (4), any monomial in the generators can be written to contain powers $(y_{\overline{w}})^n$ for at most one $\overline{w}\in\boldsymbol{D}_{\om,\operatorname{np}}^{(z-1)}/\Z$, and using relation (2) we can make it so there is at most one term of the form $u_\xi$, $\xi\in \C^\times$. If no terms like $(y_{\overline{w}})^n$ are present, then for some $\xi\in\C^\times$ and some $a\geq 0$, the monomial is of the form 
$$u_\xi u(1,2)^a = [V^1(\om,1^p,(x-\xi)^a)]+ \text{lower terms}$$ 
where the lower terms are a linear combination of classes of the form $[V^1(\om,1^p,(x-\xi)^b)]$, with $b<a$.

Otherwise, if a term $y_{\overline{w}}$ is in the monomial, with the length of $w$ being $rp$, then we let $U(r)=\{\xi\in\C^\times~|~\xi^r=1\}$ be the subgroup of $r$-th roots of unity, and we pick a fixed set of representatives for elements in the quotient group $\C^\times/U(r)$. Then we can use relation (3) to write the monomial in a unique way as $u_\xi u(1,2)^a (y_{\overline{w}})^n$, $\xi\in \C^\times/U(r).$

The map $\C^\times/U(r)\to \C^\times$ given by raising to the $r$-th power is a bijection, and we have
$$u_\xi u(1,2)^a (y_{\overline{w}})^n=[V^{(z-1)^n}(\om,w,(x-\xi^r)^a)]+ \text{lower terms}$$
where the lower terms are a linear combination of classes of the form $[V^{(z-1)^n}(\om,w,(x-\xi^r)^b)]$, with $b<a$.

It follows that all these different monomials correspond to classes of non-isomorphic indecomposable weight modules not in the ideal $\mathscr{I}(\om,\mathsf{M}_0)$ (plus lower terms), which are therefore linearly independent in $\mathscr{A}^{\mathrm{split}}(\omega,\mathsf{M}_0)/\mathscr{I}(\om,\mathsf{M}_0)$. This concludes the proof.
\end{proof}

\section{Graph modules}\label{sec:graph-mod}

In this section we consider a different approach to the problem of describing tensor products of weight modules.
The tensor product of two indecomposable weight modules from Definition \ref{def:type1mod} is in general not indecomposable. 
We consider certain graphs $(S,E)$ (with extra structure) and construct weight modules over $A(t)$ from them.
The class of these graph modules is closed under direct sums and tensor products. The corresponding Grothendieck ring is a subring of the full Grothendieck ring.

\subsection{Definition and properties}
We recall the setup from Section \ref{sec:definitions}.
Let $R$ be a commutative ring, $\si\in\Aut(R)$ be an automorphism of order $p\in\Z_{>0}$, $\om\in\MaxSpec(R)/\langle\si\rangle$ be a orbit of size $p$.  
In this section we also use the notation $\K_\Fm=R/\Fm$ for $\Fm\in\MaxSpec(R)$.
Let $\Fm_0\in\om$ a distinguished element and put $\K=\K_{\Fm_0}=R/\Fm_0$. Fix a regular element $t\in R$. Let $A(t)=R(\si,t)$ denote the generalized Weyl algebra.

Let $\Ga=(S,E)$ be a directed finite graph with vertex set $S$ and edge set $E\subseteq S^2$, such that each connected component of $\Ga$ is either a directed path or a directed cycle, equipped with a \emph{weight function} $\op{wt}:S\to\om$; an \emph{edge labeling} $\la:E\to\{1,x,y\}$; and a $2|E|$-tuple of \emph{scaling factors} $\al=(\al_\pm(e))_{e\in E}$
subject to the following conditions
\begin{enumerate}[{\rm (a)}]
\item $\forall (u,v)\in E:\,\op{wt}(v)=\si\big(\rm{wt}(u)\big)$, 
\item if $u\in S$ is a source then $\si^{-1}(\op{wt}(u))\in B^t_\om$,
\item if $v\in S$ is a sink then $\op{wt}(v)\in B^t_\om$,
\item $\forall e=(u,v)\in E:\,\la(e)=1\Leftrightarrow \op{wt}(u)\notin B^t_\om$ ($\Leftrightarrow t\notin\op{wt}(u)$),
\item $\forall e=(u^-,u^+)\in E:\,\al_\pm(e)\in \K_{\op{wt}(u^\pm)}\setminus\{0\}$, 
\item $\forall e=(u,v)\in E:\,\la(e)=1\Rightarrow \al_+(e)\si\big(\al_-(e)\big)=\si(t_{\op{wt}(u)})$, where $t_\Fm=t+\Fm\in R/\Fm$ for $\Fm\in\MaxSpec(R)$.
\end{enumerate}
Since some of these conditions depend on $t$, we call such $\op{wt}$, $\la$ and $\al$ \emph{$t$-adapted}. 
Thus $\Ga$ is specified completely by the quintuple $(S,E,\op{wt},\la,\al)$. Note that condition (a) forces the length of any cycle in $\Ga$ to be a multiple of $p$.
Define 
\[
V^t(\omega,\Ga)=\bigoplus_{\mf\in\om} V^t(\om,\Ga)_\mf,
\]
where $V^t(\om,\Ga)_\mf$ is the $R/\mf$-vector space with basis $\op{wt}^{-1}\big(\{\mf\}\big)$. Equip $V^t(\om,\Ga)$ with the structure of an $A(t)$-module by defining for all $u\in S$:
\begin{align*}
Xu&=\begin{cases} 
\al_+(e)v,& (u,v)=e\in E,\, \la(e)\neq y,\\
0,&\text{otherwise},
\end{cases}\\
Yv&=\begin{cases}
\al_-(e)u,& (u,v)=e\in E,\, \la(e)\neq x,\\
0,&\text{otherwise},
\end{cases}\\
ru&=(r+\op{wt}(u))u,\quad r\in R.
\end{align*}
It is straightforward to check that the defining relations for $A(t)$ are respected, making the action well-defined. Furthermore, $V^t(\omega,\Ga)_\Fm$ is the $\Fm$-weight space.

\begin{Lemma}
If $\Ga$ is the disjoint union of subgraphs $\Ga_1\sqcup \Ga_2$ (with restricted weight functions, edge labelings, and scaling factors) then $V^t(\om,\Ga)\cong V^t(\om,\Ga_1)\oplus V^t(\om,\Ga_2)$.
\end{Lemma}

\begin{proof}
Immediate by definition of the action of $X$ and $Y$.
\end{proof}

\begin{Lemma}\label{lem:graph-modules-no-breaks}
Suppose $\om$ has no breaks. Let $\Ga$ be a $t$-adapted graph, and suppose $\Ga$ has a single connected component (necessarily a cycle). Then $V^t(\om,\Ga)\cong V^t(\om,1^p,f)$ where $f(x)=x^d-\xi$ for some $\xi\in\K$. Conversely, any module of the form $V^t(\om,1^p,x^d-\xi)$ occurs this way.
\end{Lemma}

\begin{proof}
By the condition on the weight function, if $B^t_\om=\emptyset$ then the graph has neither a sink nor a source. Thus $\Ga$ is a cycle. This means that $\op{wt}$ is surjective. Pick $u\in S$, $\op{wt}(u)=\si\big(\Fm_0\big)$, where $\Fm_0\in\om$ is the distinguished point in the orbit. Let $n$ be the number of vertices in $\Ga$. Then, necessarily, $n=dp$ for some $d\in\Z_{>0}$.
Define $\xi\in \K$ by the relation $Y^{-n}u=\xi u$.
Then $V^t(\om,\Ga)\cong V^t(\om,1^p,x^d-\xi)$.
\end{proof}

\begin{Lemma}\label{lem:graph-modules-with-breaks}
Suppose $\om$ has at least one break.
\begin{enumerate}[{\rm (i)}]
\item If $\Ga$ consists of a single path, then $V^t(\om,\Ga)$ is isomorphic to an indecomposable weight module of the kind $V^t(\om,i,w)$. Conversely, any indecomposable weight module of that kind occurs this way.
\item If $\Ga$ consists of a single cycle, then $V^t(\om,\Ga)$ is isomorphic to a weight module of the kind $V^t(\om,w,f)$ where $f(x)=x^d-\xi$ for some $d\in\Z_{>0}$ and $0\neq \xi\in\K$ and $w\in\mathbf{D}_\om^t$. Conversely, any such weight module   occurs this way.
\end{enumerate}
\end{Lemma}

\begin{proof}
(i) Let $u_0\in S$ be the unique source of the directed path $\Ga$. By the conditions on the weight function, we therefore have $\op{wt}(u_0)=\si(\Fm_i)$ for some $i$, where $\Fm_i=\si^i(\Fm_0)$.
The direction of the edges in $\Ga$ give a total order on the set of vertices, and hence on the edges by $(u,v)<(u',v')$ if and only if $u<u'$. Let $(e_1<e_2<\cdots<e_n)$ be the ordered list of edges $e\in E$ for which $\la(e)\in \{x,y\}$. Define a word $w=w_1w_2\cdots w_n$ by $w_i=\la(e_i)$. Then $V^t(\om,\Ga)\cong V^t(\om,i,\hat{w})$, where $\hat{w}$ is obtained from $w$ as described in Remark \ref{rem:hat-w}.

Conversely, define $\Ga=\Ga(S,E,\op{wt},\la,\al)$ as follows.
Take the vertex set $S$ to be the basis for $V^t(\om,i,w)$ consisting of symbols $e^t_k$ as in the definition of $V^t(\om,i,w)$.
Define the edge set $E$ to be the set of pairs $(u,v)\in S^2$ where $Xu$ is a nonzero multiple of $v$ or $Yv$ is a nonzero multiple of $u$.
Define the weight function by $\op{wt}(e_k)=\si^k(\mf_0)$.
Let $e=(u,v)\in E$ and define 
\[
\la(e)=\begin{cases}
1,& Xu\neq 0,\, Yv\neq 0,\\
x,& Xu\neq 0,\, Yv=0,\\
y,& Xu=0,\, Yv\neq 0.
\end{cases}
\]
Define 
\[
\al_\pm(e)=
\begin{cases}
\si(t_{\op{wt}(u)}),&\la(e)=1, \text{ and } \pm=+,\\
1,&\text{otherwise}.
\end{cases}
\]
Then the identity map on $S$ extends to an isomorphism of $A(t)$-modules $V^t(\om,i,w)\cong V^t(\Ga,\om)$.

(ii) The proof is similar to part (i) and to Lemma \ref{lem:graph-modules-no-breaks}, but $\xi$ is the eigenvalue of $Z_1Z_2\cdots Z_{|S|}$ where each $Z_i$ is $Y^{-1}$ if defined and $X$ otherwise.
\end{proof}

\begin{Corollary} \phantom{text}
\begin{enumerate}[{\rm (i)}]
\item If $\om$ is an orbit without breaks, then the class of weight modules $V^t(\omega,\Ga)$ is the same as the class of direct sums of modules each of which is a module without breaks of the form $V^t(\om,1^p,x^d-\xi)$.
\item If $\om$ is an orbit with at least one break, then the class of weight modules $V^t(\om,\Ga)$ is the same as the class of direct sums of modules each of which is either a module of the kind $V^t(\om,i,w)$, or a module of the form $V^t(\om,w,x^d-\xi)$.
\end{enumerate}
\end{Corollary}

\subsection{Tensor products of graph modules}

Given two such graphs $\Ga=(S,E,\op{wt},\la,\al)$, $\Ga'=(S',E',\op{wt}',\la',\al')$ we define their \emph{tensor product} $\Ga\otimes\Ga'$ as follows. The vertex set is 
\[S''=\{v\otimes v'\mid v\in S,\, v'\in S',\, \op{wt}(v)=\op{wt}'(v')\}.\] The weight function $\op{wt}''$ is defined by $\op{wt}''(v\otimes v')=\op{wt}(v)=\op{wt}'(v')$. The edge set is 
\[E''=\big\{(u\otimes u', v\otimes v')\in (S'')^2\mid (u,v)\in E,\, (u',v')\in E',\, \{\la(u,v),\la'(u',v')\}\neq \{x,y\}\big\}.\]
The edge labeling is defined by
\[\la''\big((u\otimes u', v\otimes v')\big)=\la(u,v)\la'(u',v'),\]
where the product is computed in the following way (cf. Definition \ref{def:monoid}): \[1x=x1=x,\quad 1y=y1=y,\quad 1^2=1,\quad x^2=x,\quad y^2=y.\]
The products $xy$ and $yx$ never occur by definition of the edge set $E''$. Put $t''=t\cdot t'$. One checks that if $\la$ is $t$-adapted and $\la'$ is $t'$-adapted then $\la''$ is $t''$-adapted. Lastly, we define the scaling factor tuple $\al''$ by
\[
\al''_\pm(e'')=\al_\pm(e)\al'_\pm(e'),
\]
where $e=(u,v)$, $e'=(u',v')$, $e''=(u'',v'')=(u\otimes u',v\otimes v')$.
We check that $\al''$ is $t''$-adapted. 
If $\la''(e'')=1$ then necessarily $\la(e)=\la'(e')=1$ and hence
\begin{align*}
\al''_+(e'')\si\big(\al''_-(e'')\big)&=\al_+(e)\al'_+(e')\si\big(\al_-(e)\al'_-(e')\big)\\
&=\Big(\al_+(e)\si\big(\al_-(e)\big)\Big)\Big(\al'_+(e')\si\big(\al'_-(e')\big)\Big)\\
&=\si(t_{\op{wt}(u)})\si(t'_{\op{wt}'(u')})\\
&=\si(t''_{\op{wt}''(u'')}).
\end{align*}
This shows condition (f) holds. The other conditions are straightforward to verify.

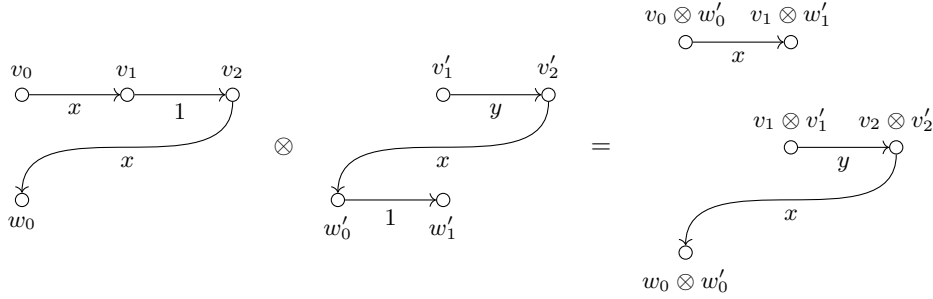
\begin{figure}[t]
\[
\begin{tikzpicture}
[scale=1.4,baseline=-.5ex,vertex/.style={draw,circle,fill=white,minimum size=5pt, inner sep=0pt}]
\node[vertex] (A) at (0,.5) [label=above:\small{$v_0$}] {};
\node[vertex] (B) at (1,.5) [label=above:\small{$v_1$}] {};
\node[vertex] (C) at (2,.5) [label=above:\small{$v_2$}] {};
\node[vertex] (D) at (0,-.5) [label=below:\small{$w_0$}] {};
\draw[->] (A) -- (B) node[midway, below] {\small $x$};
\draw[->] (B) -- (C) node[midway, below] {\small $1$};
\draw [->,out=270, in=90] (C) to node[midway, below] {\small $x$} (D);
\end{tikzpicture}
\;\;\otimes\;\;
\begin{tikzpicture}
[scale=1.4,baseline=-.5ex,vertex/.style={draw,circle,fill=white,minimum size=5pt, inner sep=0pt}]
\node[vertex] (B) at (1,.5) [label=above:\small{$v_1'$}] {};
\node[vertex] (C) at (2,.5) [label=above:\small{$v_2'$}] {};
\node[vertex] (D) at (0,-.5) [label=below:\small{$w_0'$}] {};
\node[vertex] (E) at (1,-.5) [label=below:\small{$w_1'$}] {};
\draw[->] (B) -- (C) node[midway, below] {\small{$y$}};
\draw[->,out=270, in=90] (C) to node[midway, below] {\small $x$} (D);
\draw[->] (D) -- (E) node[midway, below] {\small $1$};
\end{tikzpicture}
\;\;=\;\;
\begin{tikzpicture}
[scale=1.4,baseline=-.5ex,vertex/.style={draw,circle,fill=white,minimum size=5pt, inner sep=0pt}]
\node (X) at (0,-1) {};
\node[vertex] (A) at (0,1) [label=above:\small{$v_0\otimes w_0'$}] {};
\node[vertex] (B) at (1,1) [label=above:\small{$v_1\otimes w_1'$}] {};
\node[vertex] (C) at (1,0) [label=above:\small{$v_1\otimes v_1'$}] {};
\node[vertex] (D) at (2,0) [label=above:\small{$v_2\otimes v_2'$}] {};
\node[vertex] (E) at (0,-1) [label=below:\small{$w_0\otimes w_0'$}] {};
\draw[->] (A) -- (B) node[midway, below] {\small $x$};
\draw[->] (C) -- (D) node[midway, below] {\small $y$};
\draw[->,out=270,in=90] (D) to node[pos=0.5, below] {\small $x$} (E);
\end{tikzpicture}
\]
\caption{A tensor product $V\otimes V'$ of two graph modules. Here $\om=\{\Fm_0,\Fm_1,\Fm_2\}$; $\op{wt}(x_i)=\Fm_i$ for $x\in\{v,w,v',w'\}$, $i=0,1,2$; $B_\om^t=\{\Fm_0,\Fm_2\}, B_\om^{t'}=\{\Fm_1,\Fm_2\}$, $B_\om^{tt'}=\{\Fm_0,\Fm_1,\Fm_2\}$. The scaling factors $\alpha_\pm$ are not represented here.}
\label{fig:tensor-product-of-graph-modules}
\end{figure}

The following result, an example of which is illustrated in Figure \ref{fig:tensor-product-of-graph-modules}, is the main theorem of this section.

\begin{Theorem}
\begin{equation}
V^t(\om,\Ga)\otimes_R V^{t'}(\om,\Ga')\cong V^{tt'}(\om,\Ga\otimes\Ga').
\end{equation}
\end{Theorem}

\begin{proof}
Define a map $\varphi:V^{tt'}(\om,\Ga\otimes\Ga')\to V^t(\om,\Ga)\otimes_R V^{t'}(\om,\Ga')$ by $\varphi(v\otimes v')=v\otimes v'$ for all $v\otimes v'$ from the vertex set $S''$ of $\Ga\otimes\Ga'$. Since the weight space of weight $\Fm$ in $V^t(\om,\Ga)\otimes_R V^{t'}(\om,\Ga')$ is equal to 
$V^t(\om,\Ga)_\Fm\otimes_{\K_\Fm} V^{t'}(\om,\Ga')_\Fm$, the map $\varphi$ is an $R$-module isomorphism. Let $u''=u\otimes u'\in S''$. 
By the definition of tensor product of weight modules over generalized Weyl algebras, we have 
\[X\varphi(u'')=(Xu)\otimes (Xu').\]
Suppose first that $u\in S, u'\in S'$ such that
$e=(u,v)\in E,\, \la(e)\neq y$ and $e'=(u',v')\in E',\, \la'(e')\neq y$. Then we get
\[(Xu)\otimes (Xu')=\al_+(e)v\otimes \al'_+(e')v'=\al_+(e)\al'_+(e') v\otimes v' = \al''_+(e) v\otimes v'=\varphi(X u'').\]
If either $\la(e)=y$ or $\la'(e')=y$ both sides are zero. Similarly, $Y\varphi(u'')=\varphi(Yu'')$. Thus $\varphi$ is an $A(t'')$-module isomorphism.
\end{proof}

%

\bibliographystyle{alphaurl}

\end{document}